\documentclass{article}
\usepackage[utf8]{inputenc}
\usepackage{amsthm}
\usepackage{verbatim}
\usepackage{fancyhdr}
\usepackage{braket}
\usepackage{multirow}
\usepackage{hyperref}
\usepackage{amsmath}
\usepackage{amssymb}
\usepackage{subfiles}
\usepackage{amsthm}
\usepackage{mathrsfs}
\usepackage{amssymb}
\usepackage{tikz-cd}
\usepackage{tikz}
\usetikzlibrary{arrows,decorations.pathmorphing}
\usepackage{array}
\usepackage{polynom}
\usepackage{geometry}
\usepackage{mathtools}
\newtheorem{definition}{Definition}[section]
\newtheorem{remark}{Remark}
\newtheorem{theorem}{Theorem}
\newtheorem{lemma}{Lemma}
\newtheorem{corollary}{Corollary}
\newtheorem{proposition}{Proposition}
\newtheorem{conjecture}{Conjecture}
\newtheorem{example}{Example}
\newtheorem*{notation}{Notation}

\newtheorem{fact}{Fact}

\newcommand{\Z}{\mathbb{Z}}   \newcommand{\bbZ}{\mathbb{Z}}
\newcommand{\Q}{\mathbb{Q}}   
\newcommand{\Id}{\mathrm{Id}}

\newcommand{\Mod}{\mathrm{\textbf{Mod}}}

\newcommand{\C}{\mathbb{C}}

\newcommand{\bbC}{\mathbb{C}}
\newcommand{\cZ}{\mathcal{Z}}

\newcommand{\cB}{\mathcal{B}}

\newcommand{\cW}{\mathcal{W}}
\newcommand{\cV}{\mathcal{V}}
\newcommand{\cA}{\mathcal{A}}

\newcommand{\op}{\mathrm{op}}  \newcommand{\rev}{\mathrm{rev}}
\newcommand{\loc}{\mathrm{loc}}    \newcommand{\coop}{\mathrm{co-op}}

\newcommand{\Vect}{\mathrm{\textbf{Vec}}}

\def\sdprod{{\times\!\vrule height5pt depth0pt width0.4pt\,}}
\numberwithin{equation}{section} 
\numberwithin{theorem}{section}
\numberwithin{lemma}{section}
\numberwithin{corollary}{section}
\numberwithin{proposition}{section}

\title{Orbifolds of Pointed Vertex Operator Algebras I}
\author{
{\sc Terry Gannon} \\
 {\footnotesize Department of Mathematics, University of Alberta,}\\
{\footnotesize Edmonton, Alberta, Canada T6G 2G1}\\
{\footnotesize e-mail: {\tt tjgannon@ualberta.ca}} \\ \\
{\sc Andrew Riesen} \\
{\footnotesize Department of Mathematics, Massachusetts Institute of Technology,}\\ 
{\footnotesize Cambridge, MA 02139, USA}\\
{\footnotesize e-mail: {\tt a\_riesen@mit.edu}}}
\date{\today}
\begin{document}
\maketitle
\begin{abstract} By a \textit{pointed} vertex operator algebra (VOA) we mean one whose modules are all simple currents (i.e.\ invertible), e.g.\ lattice VOAs. This paper systematically explores the interplay between their orbifolds and tensor category theory. We begin by supplying an elementary proof of the Dijkgraaf--Witten conjecture, which predicts the representation theory of  holomorphic VOA orbifolds. We then apply that argument more generally to the situation where the automorphism subgroup fixes all VOA modules, and relate the result to recent work of Mason--Ng and Naidu. Here our results are complete. We then turn to the other extreme, where the automorphisms act fixed-point freely on the modules, and realize any possible nilpotent group as lattice VOA automorphisms. This affords a considerable generalization of  the Tambara-Yamagami categories. We conclude by  considering some hybrid actions. In this way we use tensor category theory to organize and generalize systematically several isolated examples and special cases scattered in the literature.   Conversely, we show how VOA orbifolds can be used to construct broad classes of braided crossed fusion categories and modular tensor categories.\end{abstract} 
\section{Introduction}


This paper aims to contribute to the  rich interplay between the theory of tensor categories (especially fusion categories and modular tensor categories), and the theory of vertex operator algebras (VOAs) (see e.g.\ \cite{Hu,KO}). This interplay provides  important and effective  tools in  both directions. 

A VOA  $\cV$ or a local conformal net  $\cB$ of factors on $S^1$ are both mathematical interpretations of a conformal field theory. In this paper we use the language of VOAs, since it is more familiar, but the theory of  conformal nets is more developed in some ways. The two theories are expected to be equivalent (see e.g.\ \cite{CKLW}). 

A VOA $\cV$ is called \textit{strongly rational} if its representation theory is semi-simple and finite. 
Let $G$ be a finite group of automorphisms of such a $\cV$. The fixed points $\cV^G$  are sometimes called the \textit{orbifold} of $\cV$ by $G$ (though orbifold has a different meaning in the physics literature). Taking the orbifold is one of the few general constructions we have of VOAs. Given a strongly rational $\cV$, it is conjectured (and known when $G$ is solvable) that $\cV^G$ will also be strongly rational (this is a theorem for conformal nets for all $G$ \cite{Xu}). In this paper we assume that conjecture throughout, but in examples this is rarely necessary, and everything we do can be reinterpreted in terms of conformal nets, where our results are unconditional.

This paper concerns the representation theory of the  orbifolds $\cV^G$ of strongly rational VOAs $\cV$ by finite groups $G$. We focus on the case where $\cV$ is \textit{pointed}, i.e.\ all $\cV$-modules are invertible (i.e.\ simple currents), which includes all lattice VOAs (though of course $\cV^G$ will usually have noninvertible modules). Orbifold technology is much more developed for lattice VOAs, but from the tensor category perspective there is no difference between lattice and pointed VOAs.

 In practise the most important things are the associated algebraic combinatorics. In particular, 
we want to list the simple modules of the orbifold $\cV^G$, and how their tensor (fusion) products decompose into simples. We  want to list the $G$-twisted $\cV$-modules, their fusion products, and their restrictions  to $\cV^G$-modules. We want the modular data (a matrix representation of SL$_2(\Z)$) of $\cV^G$, as this gives us its fusion decompositions, conformal weights of modules (mod 1), and modular transformations of the characters (graded dimensions). 

As we know from finite group representation theory, the complete story requires more than just the algebraic combinatorics (which for finite groups is captured by the character table). The complete story is conveniently captured by tensor categories. To give a familiar example, the dihedral group $D_4$ and the quaternions $Q_8$ have the same character table, and therefore the same  algebraic combinatorics, but their (braided fusion) categories of representations are inequivalent. More generally, the symmetric fusion category \textbf{Rep}$\,\,G$ has $G$-modules as objects and intertwiners as morphisms, and braiding $M\otimes N\to N\otimes M$ given by $u\otimes v\mapsto v\otimes u$; if the categories \textbf{Rep}$\,\,G$ and \textbf{Rep}$\,\,H$  are braided tensor equivalent, then $G$ and $H$  are isomorphic as groups.

The category \textbf{Mod}$\,\,\cV$ of modules of a VOA is defined similarly. When $\cV$ is strongly rational, \textbf{Mod}$\,\,\cV$ is of a very special type called a \textit{modular tensor category} (MTC). These are fusion categories with braidings, i.e. distinguished isomorphisms $c_{M,N}:M\otimes N\to N\otimes M$, coming from the skew-symmetry of VOA intertwiners. The tensor unit of  \textbf{Mod}$\,\,\cV$ is $\cV$ itself. The category \textbf{TwMod}$_G\,\cV$ of $G$-twisted $\cV$-modules is a braided $G$-crossed extension of \textbf{Mod}$\,\,\cV$. \textbf{Mod}$\,\,\cV$ is a full subcategory of \textbf{TwMod}$_G\,\cV$, consisting of the twisted modules with trivial grading, i.e.\ the ordinary $\cV$-modules, also called local or dyslectic. Then the MTC \textbf{Mod}$\,\,\cV^G$ is obtained from \textbf{TwMod}$_G\,\cV$ by a process called \textit{equivariantization}. (De-)equivariantization is one the few general constructions we have of tensor categories. 

Restriction along $\cV\supset\cV^G$ is a functor \textbf{Mod}$\,\,\cV\to\mathrm{\textbf{Mod}}\,\,\cV^G$, or better \textbf{TwMod}$_G\,\cV\to\mathrm{\textbf{Mod}}\,\,\cV^G$. The restriction $B=\mathrm{Res}(\cV)$ of the tensor unit can be interpreted as the algebra of functions $G\to\C$, and as such is a copy of the regular representation of $G$ with the structure of a commutative algebra. This algebra is the key to recovering \textbf{TwMod}$_G\,\cV$ and \textbf{Mod}$\,\,\cV$ from \textbf{Mod}$\,\,\cV^G$. In particular, \textbf{TwMod}$_G\,\cV$ can be interpreted as the $B$-modules in \textbf{Mod}$\,\,\cV$, and this restriction has an adjoint, the tensor functor \textit{induction} \textbf{Mod}$\,\,\cV^G\to\mathrm{\textbf{TwMod}}_G\,\cV$. 

This paper is the first in a series where we investigate the orbifolds of  VOAs.  Far from abstract nonsense, these categorical considerations provide effective tools in understanding the orbifold construction, as we will see. Tensor categories are considerably simpler, but preserve most of the salient features.  VOAs can return the favour by explicitly constructing tensor categories whose existence may not otherwise be clear (e.g. when the cohomological obstructions to their existence lie in nontrivial groups).

The most complete story, reviewed in Section 3, is when $\cV$ is \textit{holomorphic}, i.e.\
 has a trivial representation theory. Examples are lattice theories $\cV_L$ when $L$ is even, self-dual and positive definite, as well as the Moonshine module $\cV^\natural$.  
  Let $G$ be a finite group of automorphisms of such a $\cV$. The resulting tensor categories $\mathbf{Mod}\,\,\cV^G$ and $\mathbf{TwMod}_G\,\cV$ depend subtly on how $G$ acts on $\cV$, but there is only a finite  range of possibilities. Dijkgraaf--Witten \cite{DW} conjectured that the possibilities are parametrized by $[\omega]\in H^3(G,\bbC^\times)$. In particular, the $g$-twisted modules (for $g\in G$) form the   category \textbf{Vec}$_G^\omega$ of $G$-graded vector spaces (with associativity twisted by $\omega$), and the $\cV^G$-modules form the category of representations of the twisted quantum double $D^\omega(G)$. All this is now a theorem \cite{DNR} (we supply our own proof in Theorem \ref{doubleVOA}). From this the modular data for $\cV^G$ can be computed, and the restrictions and inductions between $\cV^G$-modules and $g$-twisted modules can be written down. We now know that any 3-cocycle $\omega$  (sometimes called the gauge anomaly) for any $G$ can be realized by a holomorphic VOA orbifold (Theorem \ref{dw}).
  
 Our first main task is to  generalize that story to any pointed VOA $\cV$, where $G$ fixes all $\cV$-modules (up to equivalence). We obtain that full generalization in Section 4 and the Appendix, identifying the categories of twisted $\cV$-modules and of $\cV^G$-modules, the modular data of $\cV^G$, the quantum groups corresponding to $\cV^G$ (namely the quasi-Hopf algebras of Naidu \cite{Naidu} and Mason--Ng \cite{MN}), and concrete VOA orbifolds realizing all abstract categorical possibilities.
 
 Next, in Section 5 we turn to the other extreme, where $G$ permutes all $\cV$-modules without any fixed points. When $G=\Z/2$ this recovers the Tambara-Yamagami categories (e.g.\ this happens with $\cV_L^+$ for lattices $L$ with $|L^*/L|$ is odd), but we are more interested in $|G|>2$ (e.g.\  the $\Z/3$ orbifold of the $D_4$ root lattice, where $\Z/3$ acts by triality). We focus on nilpotent $G$ for concreteness, constructing for each such $G$ a tower of generalized Tambara-Yamagami categories and their equivariantizations, which we realize as VOA orbifolds. Finally, in Section 6  we consider some hybrid cases, where $G$ possesses both fixed points and fixed-point free orbits.

The underlying picture is provided by the following figure, relating the MTCs of $\cV$- and $\cV^G$-modules, and the category of twisted modules. The latter is a braided $G$-crossed extension of the $\cV$-modules. The $H^3(G,\bbC^\times)$ ambiguity observed in the holomorphic orbifolds extends to  the general picture. In the tensor category picture  \cite{ENO}, there are obstructions $o_3,o_4$ to the existence of the braided $G$-crossed extensions, and this plays a large role in our story; having a VOA realization of the $G$-action means those obstructions must vanish.
$$\begin{tikzpicture}
\node (a) at (0,0) 
{\textbf{Mod}$\,\,\cV$}; \node (b) at (4,0) {\textbf{TwMod}$_G\,\cV$}; \node (c) at (0,-2) {$B\in\mathrm{\textbf{Mod}}\,\,\cV^G$}; \path[right hook->] (a) edge (b);   
 \path[<->] (a) edge (c);
  \path[->,font=\scriptsize] ([yshift=3pt]c.east) edge node [above] {Ind$\ \ \ \ \ \ \ \ \ $} ([yshift=3pt]b.south);   \path[<-,font=\scriptsize] ([yshift=-3pt]c.east) edge node [below] {\ \ \ \ \ Res} ([yshift=-3pt]b.south);
\end{tikzpicture}$$
\textbf{Figure 1: The Underlying Picture.} $G$ is a group of $\cV$-automorphisms, and $\cV^G$ is the orbifold. Res is the restriction functor of twisted $\cV$-modules to ordinary $\cV^G$-modules, and Ind is the induction tensor functor.  \textbf{Mod}$\,\,\cV$ and \textbf{Mod}$\,\,\cV^G$ are MTC, and \textbf{TwMod}$_G\,\cV$ is a braided $G$-crossed fusion category; \textbf{Mod}$\,\,\cV$ is a subcategory of \textbf{TwMod}$_G\,\cV$. The restriction of $\cV$ to \textbf{Mod}$\,\,\cV^G$ is a copy $B$ of the regular representation of $G$, and is a commutative algebra object in \textbf{Mod}$\,\,\cV^G$. \textbf{TwMod}$_G\,\cV$ is the category of $B$-modules, and \textbf{Mod}$\,\,\cV$ are the local  ones. 
\section{Background}

\subsection{Tensor categories}

For the standard introduction to tensor categories, see \cite{book}; for a pedagogical introduction see \cite{Wal}. For someone more comfortable with VOAs, it may be better to start with the next subsection and refer back to this one as needed.
 
A \textit{fusion category} $\mathbf{C}$ is a semisimple category with direct sums $x\oplus y$, tensor products (called fusions) $x\otimes y$, a tensor unit 1, duals $x^*$,  finitely many simples (up to equivalence), and finite-dimensional Hom spaces (for us, over $\bbC$). For example, associativity will be an explicit invertible map in Hom$_{\mathbf{C}}((x\otimes y)\otimes z,x\otimes(y\otimes z))$ for each choice of objects $x,y,z$, and those maps must satisfy certain coherence properties. A simple example is $\mathbf{Vec}_G$ for a finite group $G$, whose objects are finite-dimensional $G$-graded vector spaces; its Grothendieck ring is the group ring $\bbZ[G]$. Another example is the category $\mathbf{Rep}\,\,G$ of finite-dimensional representations. 

Given  a fusion category $\mathbf{C}$,  the \textit{Frobenius-Perron dimension} is the unique assignment of positive real numbers FPdim$(x)$ to each $x\in\mathbf{C}$ such that $$\mathrm{FPdim}(x\oplus y)=\mathrm{FPdim}(x)+\mathrm{FPdim}(y)\ \ ,\ \ \mathrm{FPdim}(x\otimes y)=\mathrm{FPdim}(x)\,\mathrm{FPdim}(y)\ \ \forall x,y\in\mathbf{C}$$
For example, any simple in $\mathbf{Vec}_G^\omega$ has FPdim 1, whereas for $\mathbf{Rep}\,\,G$ any $\rho$ has FPdim equal to its usual dimension. We write FPdim$(\mathbf{C}):=\sum_x\mathrm{FPdim}(x)^2$, where the sum is over all  isomorphism classes of simples in $\mathbf{C}$.

We say a fusion category is \textit{graded} by a group $G$ if we can write $\mathbf{C}=\oplus_{g\in G}\mathbf{C}_g$ such that $\mathbf{C}_g\otimes\mathbf{C}_h\subseteq\mathbf{C}_{gh}$ for all $g,h\in G$. It is straightforward to verify that the associativity isomorphisms of any fusion category $\mathbf{C}$ with a  $G$-grading  can be twisted by any cocycle $\omega\in Z^3(G,\bbC^\times)$, resulting in a potentially inequivalent fusion category. In this way one constructs $\mathbf{Vec}_G^\omega$.

In general, $x\otimes y\not\cong y\otimes x$ in a fusion category. By a \textit{braided} fusion category $\mathbf{C}$ we mean one with a choice of invertible map (called a \textit{braiding}) $c_{x,y}\in\mathrm{ Hom}_{\mathbf{C}}(x\otimes y,y\otimes x)$ for each choice of objects $x,y$, and those maps must satisfy certain coherence conditions. For example, although most $\mathbf{Vec}_G^\omega$ are not braided, $\mathbf{Rep}\,\,G$ has a braiding sending $u\otimes v\mapsto v\otimes u$ on the underlying spaces. Replacing a braiding  with its inverse $c^{\mathrm{rev}}_{y,x}=(c_{x,y})^{-1}$ $\forall x,y$ results in a potentially inequivalent braided fusion category called the \textit{reverse} $\mathbf{C}^{\mathrm{rev}}$. 

Let $\mathbf{C}$ be a braided fusion category, and $\mathbf{D}$  a fusion subcategory. By the \textit{M\"uger centralizer} we mean the full fusion subcategory with objects:
\begin{equation}
\mathcal{D}'=\{x\in \mathbf{C}:c_{y,x}\circ c_{x,y}=\Id_{x\otimes y}\ \ \forall y\in \mathbf{D}\}
\end{equation}
For example, $(\mathbf{Rep}\,\,G)'=\mathbf{Rep}\,\,G$. We say a braided fusion category $\mathbf{C}$ is \textit{nondegenerately} braided if it is centreless in the sense that $\mathbf{C}'\cong\mathbf{Vec}$. A \textit{modular tensor category} (MTC) is a nondegenerately braided fusion category with 
a \textit{ribbon structure} $\theta$ (see Definition 8.10.1 in \cite{book}). The ribbon structure makes it possible to define a categorical trace Tr$(f)$ of endomorphisms $f\in\mathrm{End}_{\mathbf{C}}(x)$ and categorical dimension dim$(x)$ of objects $x\in\mathbf{C}$.  As will be  explained in Section 2.3, all categories in this paper are \textit{pseudo-unitary}, i.e.\ dim$(x)=\pm\mathrm{FPdim}(x)$ $\forall x\in\mathbf{C}$.

Given a MTC $\mathbf{C}$, the different ribbon structures  on $\mathbf{C}$ are in bijection with the invertibles in $\mathbf{C}$. But when $\mathbf{C}$ is {pseudo-unitary}, there is a unique ribbon structure with dim$(x)=\mathrm{FPdim}(x)$ $\forall x$.  We can and will assume this preferred ribbon structure is chosen. It is also preferred by the associated VOAs.

Given any MTC, the \textit{modular data} is a matrix representation of the modular group SL$_2(\bbZ)$, with indices labelled by equivalence classes of simples of $\mathbf{C}$: $\left({0\atop 1}{-1\atop 0}\right)\mapsto S$ and   $\left({1\atop 0}{1\atop 1}\right)\mapsto T$, where $S_{x,y}$ is (up to a scalar independent of $x,y$) Tr$(c_{y,x}\circ c_{x,y})$ and $T$ is a diagonal matrix with entries  (up to a scalar independent of $x$) the ribbon twist $\theta(x)$. Then $x\otimes y\cong\oplus_zN_{x,y}^zz$ where the fusion multiplicities are given by \textit{Verlinde's formula} \begin{equation}\label{verl}N_{x,y}^z=\sum_{w}\frac{S_{x,w}S_{y,w}\overline{S_{z,w}}}{S_{1,w}}\end{equation}
where 1 denotes the tensor unit. The scalar normalizing $S$ is unique up to a sign, but we should always choose the sign so that $S$ has a strictly positive row (always possible, and the $S$ matrix for any VOA will also have that property). The scalar for $T$ is uniquely defined up to a third root of 1, but there is no preference for one of these over another.

Given any fusion category $\mathbf{C}$, its \textit{centre} or \textit{quantum double} $\cZ(\mathbf{C})$ is a MTC naturally associated with $\mathbf{C}$, which will play a large role in this paper. Its objects consists of pairs $(x,\beta)$ where $x$ is a (usually nonsimple) object of $\mathbf{C}$ with the property that $x\otimes y\cong y\otimes x$ $\forall y\in\mathbf{C}$ and the half-braiding  $\beta$ is (among other things) a choice of isomorphisms $\beta_y\in\mathrm{Hom}_{\mathbf{C}}(x\otimes y,y\otimes x)$ $\forall y$. When $\mathbf{C}$ is already nondegenerately braided, then $\cZ(\mathbf{C})\cong \mathbf{C}\boxtimes\mathbf{C}^{\mathrm{rev}}$.

Let $\mathbf{C}$ be a fusion category. An \textit{algebra} $B$ in \textbf{C} is an object $B\in  \mathbf{C}$, a multiplication $m \in\mathrm{Hom}_{\mathbf{C}}( B\otimes B,B)$ and a unit $\eta\in\mathrm{Hom}_{\mathbf{C}}(1, B)$ satisfying associativity etc. Algebras are involved in both the extension theory for VOAs and their orbifold theory, as we'll see. When $\mathbf{C}$ is a MTC, $B$ is called \textit{commutative} if $m=m\circ c_{B,B}$. A $B$-module is an object $M\in\mathbf{C}$ and a module multiplication $\mu\in\mathrm{Hom}_{\mathbf{C}}(B\otimes M,M)$ satisfying $\mu\circ(m\otimes \mathrm{Id}_M)=\mu\circ(\mathrm{Id}_B\otimes \mu)$ and $\mu\circ(\eta\otimes\mathrm{Id}_M)=\mathrm{Id}_M$. Given a commutative algebra $B$ in $\mathbf{C}$, the category of all $B$-modules is denoted $\mathbf{Rep}_{\mathbf{C}}\,B$. When $\mathbf{Rep}_{\mathbf{C}}\,B$ is  semisimple and $B$ is commutative, we  call $B$ \textit{\'etale}. These are the algebras we are interested in.  When dim$\,\mathrm{Hom}_{\mathbf{C}}(1,B)=1$ and $B$ is \'etale, $\mathbf{Rep}_{\mathbf{C}}\,B$ is naturally a fusion category.

Let $B$ be \'etale and dim$\,\mathrm{Hom}_{\mathbf{C}}(1,B)=1$. For reasons that will be clear next subsection, we call the forgetful functor $\mathbf{Rep}_{\mathbf{C}}\,B\to \mathbf{C}$ \textit{restriction}. By \textit{induction} we mean its adjoint  $\mathbf{C}\to \mathbf{Rep}_{\mathbf{C}}\,B$; it is a tensor functor. A simple $B$-module $M$ is called \textit{local} if Res$\,M$ has a well-defined twist $\theta$. The local $B$-modules form a MTC we call $\mathbf{Rep}_{\mathbf{C}}^{\mathrm{loc}}\,B$.

The special case most relevant to orbifold theory is when the MTC $\mathbf{C}$ contains a subcategory braided tensor equivalent to $\mathbf{Rep}\,\,G$ for some finite group $G$. Then the (regular representation) object $B_G=\oplus_{\chi\in\mathrm{Irr}(G)} \chi(e)\,\chi\in \mathbf{Rep}\,\,G$ has a natural \'etale algebra structure in $\mathbf{Rep}\,\,G$ and hence $\mathbf{D}$.
In this case, $\mathbf{C}_G=\mathbf{Rep}_{\mathbf{C}}\,B_G$ has the structure of a \textit{braided $G$-crossed category} and is called the \textit{de-equivariantization} of $\mathbf{C}$. This means that the fusion category $\mathbf{C}_G$ has a $G$-grading $\oplus_{g\in G}(\mathbf{C}_G)_g$, a $G$-action  with $h\in G$ sending $(\mathbf{C}_G)_g$ to $(\mathbf{C}_G)_{hgh^{-1}}$, and a $G$-crossed braiding $c_{x,y}\in \mathrm{Hom}_{\mathbf{C}_G}(x\otimes y,{}^gy\otimes x)$ where $x\in(\mathbf{C}_G)_g$. See e.g.\  \cite{MuA} for the full definition. The local $B_G$-modules can be identified with $(\mathbf{C}_G)_e$. We say $\mathbf{C}_G$ is a \textit{braided $G$-crossed extension} of $(\mathbf{C}_G)_e$. Conversely, given a faithful braided $G$-crossed category $\mathbf{D}$,  where $\mathbf{D}_e$ is a MTC (faithful means $\mathbf{D}_g\ne 0$ for all $g$), \textit{equivariantization} is the process constructing a MTC $\mathbf{D}^G$ containing $\mathbf{Rep}\,\,G$, such that $(\mathbf{D}^G)_G\cong \mathbf{D}$.

\begin{proposition}\label{rankDg} Let $\mathbf{D}=\oplus_{g\in G}\mathbf{D}_g$ be a faithful braided $G$-crossed extension of a MTC $\mathbf{C}\cong\mathbf{D}_e$. Then rank$(\mathbf{D}_g)$ equals the number of isomorphism classes $[M]$ of simples in $\mathbf{C}$ with ${}^gM\cong M$, and $\mathrm{FPdim}(\mathbf{D}_g)=\mathrm{FPdim}(\mathbf{C})$ for all $g\in G$.\end{proposition}

By rank$(\mathbf{D}_g$) we mean the number of isomorphism classes of simples in the subcategory $\mathbf{D}_g$.
For a simple proof of the first statement, see \cite{Bi1}; the second statement is Proposition 8.20 of \cite{ENO0}.  The VOA analogue is proved in Theorem 2.9(2) of \cite{DRX}.

By $\mathbf{C}\boxtimes\mathbf{D}$ we mean their \textit{Deligne product}, i.e.\ simples are pairs $(x,y)$ for $x\in\mathbf{C}$ and $y\in\mathbf{D}$ etc.
The following result (Corollary 3.30 in \cite{DMNO}) will be useful latter:
\begin{proposition} \label{Z(CG)}
Let $\mathbf{D}$ and $\mathbf{C}\cong\mathbf{D}_e$ be as in Proposition \ref{rankDg}.  Then $\cZ(\mathbf{D})$ is braided tensor equivalent to $ \mathbf{D}^G\boxtimes \mathbf{C}^{\rev}$, i.e.\  there is an injective tensor functor $\mathbf{C}^{\rev}\hookrightarrow \mathcal{Z}(\mathbf{D})$ such that $\mathbf{D}^G\cong (\mathbf{C}^{\rev})'$. 
\end{proposition}

Let $\mathbf{C}$ be a MTC. By a \textit{braided tensor autoequivalence} we mean a braided tensor functor which is also an autoequivalence of category $\mathbf{C}$. Equivalence classes of these form a group $\mathrm{EqBr}(\mathbf{C})$. This group is isomorphic to the group Pic($\mathbf{C})$ of equivalence classes of invertible $\mathbf{C}$-module categories under the operation of relative tensor product, but we will not explicitly use that here. We are interested in the case where $\mathbf{C}$ is pointed, when this group can be described much more explicitly (see Section 2.3). 

\subsection{Rational vertex operator algebras}

For an elementary introduction to vertex operator algebras (VOAs), see e.g.\ \cite{LL}. We restrict attention in this paper to
\textit{strongly rational} VOAs. By this we mean a VOA $\cV$ which is $C_2$-cofinite, regular, of CFT-type (i.e.\ decomposes into $L_0$-eigenspaces as $\cV=\coprod_{n=0}^\infty \cV_n$ where $\cV_0\cong\bbC$ and dim$\,\cV_n<\infty$), and is simple and isomorphic to its contragredient $\cV^*$ as a $\cV$-module. These conditions guarantee the richest representation theory.

Throughout this paper, by a  $\cV$-module $M$ we mean a grading-restricted ordinary module. This means $L_0$ also decomposes $M$ into a sum $\coprod_{h\in\mathbb{Q}}M_h$ of finite-dimensional spaces.
By \textbf{Mod}$\,\,\cV$ we mean the category of $\cV$-modules.

\begin{theorem} Suppose $\cV$ is strongly rational. Then $\mathbf{Mod}\,\,\cV$ has the structure of a MTC. \end{theorem}

This is due to Huang (see Theorem 4.6 of \cite{Hu}). The tensor unit is $\cV$ itself. For simple $M$, the ribbon twist $\theta(M)=e^{2\pi \mathrm{i}h_M}$ where $h_M$ is the conformal weight (the smallest $h$ with $M_h$ nonzero). If $h_M>0$ for all simple $M\not\cong\cV$ (the case in this paper) then $\mathbf{Mod}\,\,\cV$ will be pseudo-unitary. The modular data governs the modular transformations of the characters $\chi_M(\tau)=q^{-c/24}\sum_h\mathrm{dim}\,M_h\,q^h$ ($q=e^{2\pi \mathrm{i}\tau}$) but this plays no role in the paper.

The simplest strongly rational VOAs are the lattice VOAs $\cV_L$, where $L$ is even and positive definite  (see e.g.\ Sections 6.4,6.5 of \cite{LL}). They play a large role in this paper. $\cV_L$ is generated by basis elements $e^v$, $v\in L$, in a twisted group algebra of $L$, and by Heisenberg modes $h_{-n}$ for $n\in\Z_{>0}$ and $h\in \bbC\otimes_\Z L$. The $e^v$ satisfy $e^ue^v=\varepsilon(u,v)\,e^{u+v}$ for some 2-cocycle $\varepsilon\,{:}\,L\times L\to \{\pm 1\}$ obeying
\begin{eqnarray}\varepsilon(u,v)\,\varepsilon(u+v,w)=\varepsilon(v,w)\,\varepsilon(u,v+w)\,,\label{eps1}\\ 
\varepsilon(u,v)\,\varepsilon(v,u)=(-1)^{u\cdot v}\,,\\
\varepsilon(u,0)=\varepsilon(0,u)=1\,,\label{eps3}\end{eqnarray}
for all $u,v,w\in L$. As a VOA, $\cV_L$ is independent of the choice of $\varepsilon$; a convenient  choice is:
\begin{equation}\label{eps}\varepsilon(v_i,v_j)=\left\{\begin{array}{cc}(-1)^{v_i\cdot v_j}&\mathrm{if}\ i<j\\ 1&\mathrm{otherwise}\end{array}\right.\end{equation}
for any $\Z$-basis $v_1,...,v_d$ of $L$, and extend $\varepsilon$  linearly to all $u,v\in L$. This $\varepsilon$ satisfies \eqref{eps1}--\eqref{eps3}.

Given  strongly rational VOAs $\cV,\cW$, then $\cV\otimes\cW$ is also strongly rational, with $\mathbf{Mod}\,\,\cV\otimes\cW\cong\cV\boxtimes\mathbf{\cW}$.

A VOA $\cV$ is called \textit{holomorphic} if it is strongly rational and $\mathbf{Mod}\,\,\cV\cong\mathbf{Vec}$. For example, $\cV_L$ for self-dual $L$ are holomorphic, as is the Moonshine module $\cV^\natural$.

The theory of VOA extensions is entirely categorical \cite{KO,HKL,CKM}. If $\cW\subseteq\cV$ are both strongly rational and share the same conformal vector, then $\cV$ regarded as a $\cW$-module has the structure of an \'etale algebra $B$  in $\mathbf{Mod}\,\,\cW$ with dim$\,\mathrm{Hom}_{\mathbf{Mod}\,\,\cW}(1,B)=1$. $\mathbf{Mod}\,\,\cV$ will be the subcategory of local modules in the fusion category $\mathbf{Rep}_{\mathbf{Mod}\,\,\cW}\,B$. Moreover, $\cZ(\mathbf{Rep}_{\mathbf{Mod}\,\,\cW}\,B)$ is braided tensor equivalent to $\mathbf{Mod}\,\,\cW\boxtimes(\mathbf{Mod}\,\,\cV)^{\mathrm{rev}}$. Conversely, given any \'etale algebra $B$ in a pseudo-unitary $\mathbf{Mod}\,\,\cW$ with dim$\,\mathrm{Hom}_{\mathbf{Mod}\,\,\cW}(1,B)=1$, there is a strongly rational $\cV\supseteq\cW$ such that $B=\mathrm{Res}_\cW\,\cV$ (pseudo-unitarity here is only needed to ensure $\theta(B)=1$, otherwise we'd need to assume that).

By an \textit{automorphism} $g$ of a VOA $\cV$, we mean a linear isomorphism $g\,{:}\,\cV\to\cV$ satisfying $g(u_nv)=(g(u))_ng(v)$ and fixing the conformal vector $\omega$. We will discuss the automorphisms of $\cV_L$ in Section 2.3. 
Given a VOA $\cV$ and a finite group $G$ of automorphisms, by the fixed-point VOA (often called \textit{orbifold}) $\cV^G$ we mean the set of all $v\in\cV$ fixed by all $g\in G$. Then $\cV^G$ will always be a VOA.

The theory of VOA orbifolds is largely categorical, though whether a given VOA has a group of automorphisms isomorphic to a given group $G$,  depends on the VOA and not its category.

\begin{conjecture} \label{conj} Suppose $\cV$ is strongly rational and $G$ is a finite group of automorphisms. Then $\cV^G$ is also strongly rational. \end{conjecture}

This is proven for $G$ solvable (see Main Theorem 2 in \cite{McR} and  Corollary 5.25 in \cite{CM}). 

Given a $\cV$-automorphism $g$, there is a notion of  \textit{$g$-twisted $\cV$-module}.  Their definition  is not important for us because of  Proposition 4.13 of \cite{McR}. More precisely, 
since  $\cV^G\subseteq\cV$, there is an \'etale algebra $B_G=\mathrm{Res}_{\cV^G}\,\cV\in\mathbf{Mod}\,\,\cV^G$. McRae showed that the notion of $g$-twisted $\cV$-module for some $g\in G$ coincides with that of $B_G$-modules in $\mathbf{Mod}\,\,\cV^G$.  Denote the category of (direct sums of) $g$-twisted modules for $g\in G$ by $\mathbf{TwMod}_G\,\mathcal{V}$. Then it coincides with $\mathbf{Rep}_{\Mod \, \mathcal{V}^G} \, B_G$.  Furthermore:
\begin{theorem} \label{voathm} Suppose $\cV$ is strongly rational. Let $G$ be a finite group of automorphisms of $\cV$. Assume that the orbifold $\cV^G$ is strongly rational. Then the category  $\mathbf{TwMod}_G\,\cV$ of $G$-twisted $\cV$-modules form a braided $G$-crossed extension of $\mathbf{Mod}\,\,\cV$, tensor equivalent to $\mathbf{Rep}_{\mathbf{Mod} \, \mathcal{V}^G} \, B_G$. The category $\mathbf{Mod}\,\,\cV^G$ is braided tensor equivalent to the equivariantization  $(\mathbf{TwMod}_G\,\cV)^G$.\end{theorem}

This is Theorems 4.15,4.17 of \cite{McR}  (see also \cite{Kir0,Kir}). Actually, McRae works in far greater generality than expressed here.
The  $\cV$-automorphisms which fix $\cV^G$-pointwise form the set of all invertible $\alpha\in\mathrm{End}_{\mathbf{Mod}\,\,\cV}(B_G)$ satisfying $m\circ(\alpha\otimes\alpha)=m$ and $\alpha\circ\eta=\eta$, and this by Proposition 3.2(iv) of \cite{MuA} is isomorphic to $G$. So the tensor category story matches the VOA one.

\begin{remark} To get Theorem \ref{voathm}, the tensor product
on $\mathbf{TwMod}_G\,\cV$  is defined in \cite{McR} through what they call twisted intertwining operators (see \cite[Section 3.5]{CKM} for more details). When $G$ is finite and abelian, it is known \cite[Remark 4.16]{McR} that these will coincide with the twisted intertwining operators first defined by Xu in \cite{XuX}. To our knowledge, it is currently unknown if their twisted intertwining operators always coincide with the twisted intertwining operators defined by Huang in \cite{Hua1}. See Sections $5$  and $6$ of \cite{Hua2} for a more in-depth discussion. This gap in the VOA literature plays no role in this paper.
\end{remark}

Given an automorphism group $G$ of $\cV$, $G$ will permute the equivalence classes of $\cV$-modules, and in fact act as (possibly trivial) automorphisms of the fusion ring of $\mathbf{Mod}\,\,\cV$. This is part of the $G$-action on $\mathbf{Rep}_{\cV^G}\,B_G\cong\mathbf{TwMod}_G\,\cV$, restricted to $(\mathbf{TwMod}_G\,\cV)_e$. We call this the \textit{module-map}.

The structural theory for strongly rational conformal nets $\cA$ is somewhat easier. That their representations form a
 (unitary) MTC, is established in Section 5 of \cite{KLM}. That $\cA^G$ is also strongly rational, for any finite $G$, is Theorem 2.6 in \cite{Xu}. That the $G$-twisted representations (called \textit{solitons} in the conformal net literature) form a braided $G$-crossed category, is Theorem 2.21 in \cite{Mu}. That \textbf{Rep}$\,\,\cA^G$ is braided tensor equivalent to the equivariantization $(\mathrm{TwRep}_G\,\cA)^G$, is Theorem 3.12 of \cite{Mu}.

\subsection{Pointed MTCs and pointed VOAs}

Suppose a fusion category $\mathbf{C}$  is pointed, i.e. all of its simple objects are invertible, or equivalently its fusion (or Grothendieck) ring is $\Z G$ for some finite group $G$. Then $\mathbf{C}$ is tensor equivalent to $\mathbf{Vec}_G^\omega$ for some normalized 3-cocycle $\omega\in Z^3(G,\bbC^\times)$ (normalized means $\omega(e,h,k)=\omega(g,e,k)=\omega(g,h,e)=1$ $\forall g,h,k\in G$). Moreover,  $\mathbf{Vec}_G^\omega$ is tensor equivalent to $\mathbf{Vec}_H^{\omega'}$ only if $G\cong H$ as groups, and  $\mathbf{Vec}_G^\omega$ is tensor equivalent to $\mathbf{Vec}_G^{\omega'}$ iff $[\omega']=[\alpha^*\omega]$ in $H^3(G,\bbC^\times)$ for some outer automorphism $\alpha$ of $G$ (Proposition 2.6.1(iii) of \cite{book}).

Suppose now an MTC $\mathbf{C}$  is pointed. Then the braiding on $\mathbf{C}$ forces the group  to be abelian. Let $A$ be any finite abelian group. The MTCs $\mathbf{C}$ whose fusion ring is $\Z A$ can be parametrized in two different ways: in terms of \textit{metric groups} $(A,q)$ where $q:A\to\C^\times$ is a non-degenerate quadratic form, or in terms of \textit{abelian 3-cocycles} $(\omega,c)$. The quadratic form description directly connects to the modular data -- in particular, the ribbon twist is $\theta(x)=q(x)\Id_x$, and the $S$-matrix comes from the associated bicharacter of $q$. The equivalence of these two pictures is explained in Section 8.4 of \cite{book} -- e.g.\ $q(a)=c_{a,a}$. We will denote this modular tensor category \textbf{Vec}$_A^{(\omega,c)}=\mathrm{\textbf{Vec}}_A^q$. Exercise 8.4.8 of \cite{book} says when \textbf{Vec}$_A^{(\omega,c)}$ and \textbf{Vec}$_A^{(\omega',c')}$ are braided tensor equivalent.

The pairs $(\omega,c)$ are not arbitrary, they are subject to compatibility equations (see (8.11) in \cite{book}). For example, up to equivalence, $A=\Z/2$ we can take $\omega(1,1,1)=-1$ and $c(1,1)=\pm\mathrm{i}$ (the other values are all 1). For $|A|$ odd, we can take $\omega$ identically 1, and $c$ to be any nondegenerate bicharacter on $A$.

Incidentally, $(\mathbf{Vec}_A^q)^{\mathrm{rev}}=\mathbf{Vec}_A^{\bar{q}}$ and $\mathbf{Vec}_{A_1}^{q_1}\boxtimes\mathbf{Vec}_{A_2}^{q_2}=\mathbf{Vec}_{A_1\oplus A_2}^{q_1\oplus q_2}$.

In this context, $\omega$ is often cohomologically trivial. More precisely, write $A=A_e\oplus A_o$ where $A_e$ is a 2-group and $A_o$ has odd order, and write $\omega_e$ and $\omega_o$ for the restriction of $\omega$ to $A_e$ resp.\ $A_o$. Then $\omega=\omega_e\omega_o$, and the existence of a braiding forces $[\omega_o]=[1]$, i.e.\ $\omega_o$ is cohomologically trivial. $\omega_e$ will  be cohomologically trivial iff the order of each element $x\in A_e$ is a multiple of the order of its root of unity $q(x)$. When $A_e$ is cyclic, non-degeneracy  of the braiding forces $[\omega_e]$ to be nontrivial.

We say a fusion category is \textit{weakly integral} if FPdim$(\mathbf{C})\in\bbZ$, and that it is \textit{integral} if FPdim$(\mathbf{x})\in\bbZ$ for all objects $x$. Any pointed category is integral.
Proposition 2.18 of \cite{DGNO} says that, since \textbf{C} is (weakly) integral, any braided $G$-crossed extension \textbf{D} will be weakly integral and hence pseudo-unitary, as will its equivariantization \textbf{D}$^G$. Hence we can (and will) choose spherical structures  so that categorical dimensions coincide with Frobenius--Perron dimensions (and are thus positive). In a weakly integral category, each $(\mathrm{FPdim}(x))^2$ will be a positive integer.

The group EqBr($\textbf{C}$) of braided tensor autoequivalences of a pointed modular category $\textbf{C}=\textbf{Vec}_A^q$ is naturally isomorphic to O$(A, q)$, the group of group automorphisms of $A$ that fix $q$.

By a \textit{pointed VOA} we mean a strongly rational VOA $\cV$ whose \textbf{Mod}$\,\,\cV$  is pointed. The simplest examples of pointed VOAs are the lattices ones $\cV_L$, for $L$ even and positive definite. In this case, \textbf{Mod}$\,\,\cV_L=\mathbf{Vec}_A^q$ where $A=L^*/L$ and $q([v])=e^{\pi\mathrm{ i} v\cdot v}$. Given any pointed MTC \textbf{C}, there are even positive definite lattices $L$ such that \textbf{Mod}$\,\,\cV_L\cong\,$\textbf{C} as a MTC (see e.g.\ Theorem 2 in \cite{EGty}). However pointed VOAs certainly need not be lattice VOAs (e.g.\ consider the Moonshine module $\cV^\natural$, or its $\bbZ/N$-orbifolds).

 According to \cite{DN}, the automorphisms  of the VOA $\cV_L$ are generated by isometries of the lattice (i.e. $g:L\to L$ is linear and preserves the inner product on $L$), together with automorphisms of the form $\alpha_x$ for $x\in \bbC \otimes_\bbZ L$. An isometry $\sigma$ of $L$ sends each $h\in\bbC\otimes_\Z L$ to $\sigma(h)$ and $e^v$ to $\eta_\sigma(v)\,e^{\sigma(v)}$ where $\eta_\sigma:L\to\{\pm 1\}$ satisfies
 \begin{equation}\label{eta} \frac{\eta_\sigma(u+v)}{\eta_\sigma(u)\,\eta_\sigma(v)}=\frac{\varepsilon(\sigma(u),\sigma(v))}{\varepsilon(u,v)}\,. \end{equation}
 $\alpha_x$ fixes the $h$ and sends $e^v$ to $e^{2\pi \mathrm{i}x\cdot v}e^v$. We are primarily interested in the $\cV_L$-automorphisms coming from isometries. Any $L$-isometry $\sigma$ has a lift $\widehat{\sigma}$ satisfying  $\widehat{\sigma}(e^v)=e^v$ for all $v\in L$ fixed by $\sigma$; this lift is unique (up to conjugation in Aut$(\cV_L)$) and is called the standard lift.

Let $\sigma$ be an isometry of a lattice $L$. Then its {standard lift} has the same order as $\sigma$, or twice that order. Propositions 7.3 and 7.4  of \cite{EMS} say:

\begin{lemma} \label{voaaut}Let $L$ be an even positive definite lattice, and $\sigma$ be an isometry of $L$ of order $n$. Let $\widehat{\sigma}$ be a standard lift to $\mathrm{Aut}(\cV_L)$. Then the order of $\widehat{\sigma}$ is also $n$, if either $n$ is odd, or if $u\cdot \sigma^{n/2}(u)$ is even for all $u\in L$.\end{lemma}

Given an automorphism group $G$ of a pointed VOA $\cV$, we get a group homomorphism $\rho:G\to \mathrm{O}(A,q)$. The module-map is the corresponding permutation action of $G$ on the modules, i.e.\  $\rho:G\to\mathrm{Aut}(A)$.

\subsection{Realization by VOAs}

The main theme of this paper is the interplay between tensor categories and VOAs. Part of this is \textit{VOA reconstruction}: given a MTC $\mathbf{C}$, show that there exists a strongly rational VOA $\cV$ such that \textbf{Mod}$\,\,\cV\cong\mathbf{C}$ (or show that no such VOA can exist). 
A basic result, mentioned last subsection, is that any pointed MTC can be realized by a lattice VOA. Far less trivial  is that any twisted quantum double of a finite group can be realized as a VOA (Corollary 4 of \cite{EG}).

One value of realizing a MTC as a lattice VOA orbifold, is that much of the modular data can be accessible, through theta function arguments. See \cite{BEKT} for a systematic treatment, when orbifolding the lattice VOA by a cyclic group $G$. There are limitations however -- e.g.\ the theta functions in general won't be linearly independent. But when used in conjunction with categorical arguments, it can provide enough information to pin down exactly where on the $H^2$- and $H^3$-torsors the given orbifold lies (these torsors are described next subsection).

The following  illustrates a general strategy for proving VOA reconstruction:
\begin{theorem}
\label{reconstruction_theorem}
Let $\mathbf{D}$ be a MTC. Suppose that $B$ is an \'etale algebra in $\mathbf{D}$. If $\mathcal{Z}(\mathbf{Rep}_{\mathbf{D}}\, B)\cong \mathbf{Mod}\,\, \mathcal{W}$ for a strongly rational VOA $\mathcal{W}$ and $\mathbf{Rep}^{\loc}_{\mathbf{D}}\, B\cong \mathbf{Mod}\ \mathcal{V}, $ for a strongly rational VOA $\mathcal{V}$, then there exists a strongly rational VOA $\mathcal{U}$ such that $\mathbf{Mod}\,\, \mathcal{U}\cong \mathbf{D}$.
\end{theorem}
\begin{proof} Write $\mathbf{C}= \mathbf{Rep}^{\loc}_{\mathbf{D}}\, B$; by assumption it is an MTC. 
First, recall that by Proposition \ref{Z(CG)} we know that 
\[\hspace{-0.25cm}\mathbf{Mod}\,( \mathcal{W}\otimes \mathcal{V})=\mathbf{Mod}\,\, \mathcal{W}\boxtimes\mathbf{Mod}\,\, \mathcal{V}\cong 
\mathcal{Z}(\mathrm{\textbf{Rep}}_{\mathbf{D}}\, B)\boxtimes \mathbf{C}\cong (
\mathbf{D}\boxtimes \mathbf{C}^{\mathrm{rev}})\boxtimes \mathbf{C}\cong \mathbf{D}
\boxtimes \mathcal{Z}(\mathbf{C})
\]
Applying \cite[Proposition 7.1]{FFRS} we know there exists a $\mathcal{Z}(\mathbf{ C})$-algebra $T$ such that $\mathrm{\textbf{Rep}}_{\mathcal{Z}(\mathbf{ C})}^{\loc}\, T \cong \mathbf{Vec}$.
On the other hand,  $\cW\otimes\cV$ is also strongly rational; we see that $1\boxtimes T$ will be an \'etale algebra in $\mathbf{D}
\boxtimes \mathcal{Z}(\mathbf{C})\cong\mathrm{\textbf{Mod}}(\mathcal{W}\otimes \mathcal{V})$. By VOA extension theory, we see that $\mathbf{D}\cong \mathrm{\textbf{Mod}}\,\,\mathcal{U}$ for the VOA extension $\mathcal{U}$ of $\mathcal{W}\otimes \mathcal{V}$ corresponding to $1\boxtimes T$. 
\end{proof}

For the applications we have in mind this paper, $B=B_G$ is a copy of the regular representation of $G$ in $\mathbf{D}$, so $ \mathbf{Rep}_{\mathbf{D}}\, B$ is a braided $G$-crossed extension of $ \mathbf{Rep}^{\loc}_{\mathbf{D}}\, B$.
However, Theorem \ref{reconstruction_theorem} doesn't go as far as we would like. We would like to realize $\mathcal{U}$ as a $G$-orbifold, i.e.\ find a strongly rational VOA $\mathcal{U}'$, on which $G$ acts by VOA-automorphisms, such that $\mathbf{Mod}\,\,\mathcal{U}^{\prime\, G}\cong\mathbf{D}$ and $\mathbf{TwMod}_G\,\mathcal{U}'\cong\mathbf{Rep}_{\mathcal{D}}\,B$. Compare Theorem \ref{reconstruction_theorem} with the stronger Theorems \ref{doubleVOA} and \ref{recon}, Corollaries \ref{recon_gpth} and \ref{TYeven}, and Lemma \ref{L:fpf} below.

\subsection{Etingof-Nikshych-Ostrik gauging theory}

The categorical interpretation of the orbifold construction is called \textit{gauging}. Let $\textbf{C}$ be a MTC. Suppose we are given an  action of $G$ on $\textbf{C}$ which respect the braiding, i.e. a group homomorphism $\rho : G \to \mathrm{EqBr}(\textbf{C})$. To gauge this action, we first require a compatible braided $G$-crossed extension $\textbf{D}$ of $\textbf{C}$. This may not exist, and if it does it is usually not unique. Given any such $\mathbf{D}$, we can then equivariantize by the associated action of $G$, to produces a new MTC $\mathbf{D}^G$ called a gauging of $\textbf{C}$ by $G$. If $\mathbf{Mod}\,\,\cV\cong\mathbf{C}$, then $\mathbf{Mod}\,\,\cV^G$ will be braided tensor equivalent to some $\mathbf{D}^G$.

However, associated to that initial homomorphism $\rho : G \to \mathrm{EqBr}(\mathbf{C})$ are two cohomological obstructions to finding such a  braided $G$-crossed category \cite{ENO} (see also \cite{CGPW,DaN}). The first, $o_3(\rho)$,  lives in $H^3(G,\mathrm{Inv}(\mathbf{C}))$, where Inv$(\mathbf{C})$ are the invertible objects (or simple currents) in $\mathbf{C}$. $o_3(\rho)$ vanishes iff the group homomorphism $\rho$ can be lifted into a categorical action $\underline{\rho}$, in which case the possible liftings form a torsor over $H^2(G,\mathrm{Inv}(\textbf{C}))$. Now choose any such lifting $\underline{\rho}$. It in turn can be lifted to a braided $G$-crossed extension of $\mathbf{C}$, iff another obstruction $o_4(\underline{\rho})\in H^4(G,\bbC^\times)$ vanishes. If $o_4(\underline{\rho})$ also vanishes, then the possible braided $G$-crossed extensions of $\mathbf{C}$ coming from $\underline{\rho}$ form a torsor over $H^3(G,\bbC^\times)$. Incidentally, this $H^3$-torsor is the twisting of associativity in graded categories, described in Section 2.1.

 As we will see, it is often relevant that  Eilenberg-Mac Lane theory (see e.g. Chapter IV of \cite{Mac}) uses $H^2(G,A)$ to describe the  possible abelian extensions $\Gamma$ of $G$ by $A$:
\begin{equation} 0\to A\hookrightarrow\Gamma\to G\to 1\label{gpext}\end{equation}
(Throughout the paper we  use additive notation for $A$ and multiplicative for $G$.)
The interpretation of the $H^2$ torsor using zesting is explored in \cite{zest} though doesn't play a role here.

Suppose we are given any pointed VOA $\cV$, with \textbf{Mod}$\,\,\cV\cong\mathbf{Vec}_A^q$, and some group $G$ of automorphisms of $\cV$. Then this corresponds to some orthogonal map $\rho\in \mathrm{O}(A,q)$. We know, by Theorem \ref{voathm}, that there is an associated braided $G$-crossed extension of $\mathbf{Vec}_A^q$. Thus the first obstruction $o_3(\rho)$ must vanish, as must the second, $o_4(\underline{\rho})$ for some lift of $\rho$.
\textit{This is the main value to tensor categories of VOA orbifold theory:} the existence of certain tensor categories, even when the cohomology groups in which the obstructions live are large.

For future reference, if $|G|$ is coprime to $|A|$, then $H^n(G,A)=0$ for all $n>0$. When $G=\Z/n$ acts trivially on $A$, then $H^2(\Z/n,A)\cong A/nA$ and $H^3(\Z/n,A)$ consists of the elements of $A$ of order dividing $n$. Also, $H^3(\Z/n,\bbC^\times)\cong\Z/n$ and  $H^4(\Z/n,\bbC^\times)=0$. 

\subsection{Realizations by Hopf-like algebras}

Any fusion category is tensor equivalent to the category of finite-dimensional representations of a weak Hopf algebra (Corollary 2.22 of \cite{ENO0}). In a weak Hopf algebra, the comultiplication may not be unit-preserving and the counit may not be an algebra homomorphism. Weak Hopf algebras are much less easy to work with than Hopf algebras. 

In order for the representations of a (weak) Hopf algebra to have a braiding, the extra structure needed is an $R$ matrix, and such algebras are called quasitriangular. If their category also possesses a ribbon structure (needed to recover an MTC), they are called ribbon.

When the fusion category is integral, then it is equivalent to the representation category of a finite-dimensional quasi-Hopf algebra (Proposition 6.1.14 of \cite{book}).  Quasi-Hopf algebras are much more amenable than weak Hopf algebras. For example, \textbf{Vec}$_G^\omega$ is realized by the commutative quasi-Hopf algebra of $\bbC$-valued functions on $G$ twisted by $\omega$ (see e.g.\ Example 5.13.6 in \cite{book}).

\section{Warm-up: Holomorphic orbifolds}

In this section we consider  the baby case where $\cV$ is holomorphic, i.e.\ \textbf{Mod}$\,\,\cV$ is \textbf{Vec}. Dijkgraaf--Witten \cite{DW} describe how the corresponding topological field theory should look, and \cite{DPR} constructed the quasi-Hopf algebra $D^\omega(G)$ whose representations captured that theory. (More precisely, the finite-dimensional representations of  $D^\omega(G)$ form the category  $\cZ(\mathbf{Vec}_G^{\omega^{-1}})$.) By the \textit{Dijkgraaf--Witten conjecture} we mean the statement that \textbf{Mod}$\,\,V^G$ is tensor equivalent to $\cZ(\mathbf{Vec}_G^\omega)$ for some $\omega\in Z^3(G,\bbC^\times)$. This conjecture was recently proved (Theorem 6.2 of \cite{DNR}), after a special case was established much earlier in \cite{Kir0}. We supply a simpler proof in this section (in fact we prove a little more), and next section apply the argument in a much more general setting.

In the physics literature, the 3-cocycle $\omega\in Z^3(G,\bbC^\times)$ is often called the \textit{gauge anomaly}, but this seems a bit of a misnomer. It is generally nontrivial. For example, in Monstrous Moonshine (where $\cV=\cV^\natural$ is the moonshine module and $G$ is the Monster finite group $\mathbb{M}$), $[\omega]$ has order 24 in $H^3(\mathbb{M},\bbC^\times)$ \cite{JF}. For another example, $G=\bbZ/2$ acts on the $E_8$ root lattice VOA in two inequivalent ways: one gives $\cV_{D_8}$ and the other gives $\cV_{A_1\oplus E_7}$. These correspond to $[\omega]$  trivial resp.\ nontrivial in $H^3(\bbZ/2,\bbC^\times)\cong\bbZ/2$.

\subsection{In and around the Dijkgraaf-Witten conjecture}

We begin with our proof of the Dijkgraaf-Witten conjecture.
\begin{theorem} \label{dw} Let $\cV$ be a holomorphic VOA, and $G$ a finite group of automorphisms of $\cV$. Assume the fixed-point VOA $\cV^G$ is strongly rational. Then the category $\mathbf{TwMod}_G\,\cV$ of twisted modules is tensor equivalent to $\mathbf{Vec}_G^\omega$ for some $\omega\in Z^3(G,\bbC^\times)$, and $\mathbf{Mod}\,\,\cV^G\cong\cZ(\mathbf{Vec}_G^\omega)$.\end{theorem}

\begin{proof}  
 By Proposition \ref{rankDg} above, for each $g\in G$ there is a unique simple $g$-twisted module. Denote it by $M^g$, and note $\mathrm{FPdim}(M^g)=1$. This uniqueness applied to the $G$ grading forces the fusions $M^g\otimes M^h\cong M^{gh}$, so \textbf{TwMod}$_G\,\cV$ is pointed. This means that as a fusion category, \textbf{TwMod}$_G\,\cV$ must be tensor equivalent to \textbf{Vec}$^\omega_H$ for some group $H$ and $[\omega]\in H^3(H,\C^\times)$. Of course the two gradings must match, again by uniqueness, which forces $H\cong G$. So \textbf{TwMod}$_G\,\cV$ is \textbf{Vec}$^\omega_G$.  Then Proposition \ref{Z(CG)} in combination with Theorem \ref{voathm} forces
$$\cZ(\mathbf{Vec}_G^\omega)\cong\mathbf{Mod}\,\,\cV^G\boxtimes (\mathbf{Mod}\,\,\cV)^{\rev}\cong \mathbf{Mod}\,\,\cV^G.$$
\end{proof}

Up to isomorphism, the simple objects of $\cZ(\mathbf{Vec}_G^\omega)$ are parametrized by pairs $[g,\chi]$ where $g$ is the representative of a conjugacy class of $G$,  and $\chi$ runs through the irreducible projective characters Irr$_{\theta_g}(C_G(g))$ of  the centralizer $C_G(g)$, for the 2-cocycle 
\begin{equation}\label{thetaform}\theta_g(h,k)=\frac{{\omega}(g,h,k)\,{\omega}(h,k,g)}{{\omega}(h,g,k)}\ \ \forall h,k\in C_G(g)\end{equation}
Letting $M^g$ denote the unique simple $g$-twisted module, we have $$\mathrm{Res}_{\cV^G}\,M^g\cong\sum_{\chi\in\mathrm{Irr}_{\theta_g}(C_G(g))}\mathrm{dim}\,\chi\,\,[g,\chi]$$ as is well-known, hence  induction (its adjoint functor) Ind$\,[g,\chi]=\sum_{g'\in K_g}\mathrm{dim}\,\chi\,M^{g'}$, where the sum is over the elements in the conjugacy class of $g$.

Now turn to VOA reconstruction. The hard part, finding a VOA $\cV$ with $\mathbf{Mod}\,\,\cV\cong\cZ(\mathbf{Vec}_G^\omega)$,  was accomplished in Corollary 4 of \cite{EG}. That result, and the following one, assumes Conjecture  \ref{conj} in the special case of $\cV$ holomorphic. If $G$ is solvable, the conjecture is unnecessary. If one replaces VOA in the statement of the Theorem with conformal net of factors on $S^1$, the conjecture is unnecessary for any $G$.

\begin{theorem} \label{doubleVOA} Assume Conjecture \ref{conj}. For any finite group $G$ and any cocycle $\omega\in Z^3(G,\C^\times)$, a holomorphic VOA $\cV$ can be found such that $G$ acts on $\cV$ as VOA automorphisms and both $\mathbf{Mod}\,\,\cV^G\cong \cZ(\mathbf{Vec}_G^\omega)$ and $\mathbf{{TwMod}}_G\,\cV\cong\mathbf{Vec}_G^\omega$.\end{theorem}

\begin{proof} Corollary 4 of \cite{EG} tells us there exists a strongly rational  VOA $\cW$ such that $\mathbf{Mod}\,\,\cW\cong\cZ(\mathbf{Vec}_G^\omega)$. Moreover, the conformal weights $h_M$ are positive for all simple $\cW$-modules $M\not\cong\cW$.  In any $\cZ(\mathbf{Vec}_G^\omega)$, there is  an \'etale algebra structure on $B_G=\oplus_{\chi\in\mathrm{Irr}(G)}\mathrm{dim}(\chi)\,[e,\chi]$  (see e.g.\ Remark 3.2(b) of \cite{drinfeld}). It is Lagrangian, in the sense that $\mathbf{Rep}^{\mathrm{loc}}_{\cZ(\mathbf{Vec}_G^\omega)}\,B\cong\mathbf{Vec}$. This then corresponds to  a holomorphic VOA extension $\cV$ of $\cW$ with  $\mathbf{Res}_\cW\,\cV=B_G$.

How $G$ acts on $\mathcal{V}$ is given by that formula for $B_G$. By Proposition 3.2(iv) of \cite{MuA}, this action of $G$ on $\cV$ is indeed by VOA automorphisms.
In particular, the fixed points  $\mathcal{V}^G$ will be $[e,1]=\mathcal{W}$. By Theorem \ref{voathm}, $\mathbf{TwMod}_G\,\cV$ can be identified with the de-equivariantization of $\cZ(\cW)\cong\cZ(\mathbf{Vec}_G^\omega)$, i.e.\ with $\mathbf{Vec}_G^{\omega}$. \end{proof}

The ENO obstruction $o_3$ trivially vanishes here, and the torsor over $H^2$ is also trivial. The obstruction $o_4$ also must vanish, because $\mathbf{Vec}_G^\omega$ has a braided $G$-crossed structure. This can be seen using either reconstruction Theorem \ref{doubleVOA}, or de-equivariantization of $\cZ(\mathbf{Vec}_G^\omega)$. The torsor over $H^3$ corresponds to varying $\omega$.

All dimensions here are integers, so $\cZ(\mathbf{Vec}_G^\omega)$ will be the category of representations of a quasi-Hopf algebra. This twisted quantum double of $G$ was constructed in \cite{DPR}. It was through this quasi-Hopf algebra that the modular data of $\cZ(\mathbf{Vec}_G^\omega)$ was computed \cite{CGR}.

Tensor category theory has further consequences for VOAs. For example, Theorem 1 of \cite{EG} classifies all VOA extensions $\mathcal{W}$ of a holomorphic orbifold $\cV^G$ (i.e.\ all \'etale algebras in $\mathcal{Z}(\mathbf{Vec}_G^\omega)$). Corollary 2 there classifies all extensions $\cV^G\subset\mathcal{W}$ where $\mathcal{W}$ is also holomorphic. Now, the important paper \cite{EMS} addresses the classification of holomorphic VOAs at central charge $c=24$, using $\Z/n$ orbifolds of known holomorphic VOAs. Much of their Section 5 is an immediate corollary of the fact that \textbf{Mod}$\,\,\cV^{\Z/n}\cong\mathcal{Z}(\mathbf{Vec}_{\Z/n}^\omega)$. The more recent paper \cite{GK} constructs special holomorphic VOAs at $c=48$ and 72, this time using orbifolds by nonabelian (metacyclic) groups. They do apply some tensor category machinery to greatly simplify their analysis.    

\subsection{Consequences of a nontrivial twist $\omega$}

How can the gauge anomaly $\omega$ be detected?
There are several consequences of having a nontrivial  $\omega$. Many of these are discussed with examples in \cite{CGR}.

The most obvious is to the ribbon twist $\theta(M)=e^{2\pi \mathrm{i} h_M}$ for $\cV^G$-modules $M$. For any $\omega$, $\theta([g,\chi])=\frac{\chi(g)}{\chi(e)}$,  is a root of unity. When $[\omega]=[1]$ in $H^3(G,\bbC^\times)$, the order of each such root of unity will divide the exponent of $G$, but this often fails when  $[\omega]\ne[1]$. For example, when $G=\Z/n$, $H^3(G,\bbC^\times)\cong\Z/n$, which we can parametrize by $\omega_q$ for $q\in\Z/n$ as in (6.1) of  \cite{CGR}.  The simples of $\cZ(\mathbf{Vec}_{\Z/n}^{\omega_q})$ can be parametrized by $[a,\ell]$, $a\in\Z/n$, $\ell\in\widehat{\Z/n}$ for any $q$. Then
$$\theta([a,\ell])=\exp(2\pi i(qa^2+n a\ell)/n^2)$$
So if $\omega_q$ has order $n$ (i.e.\ $q$ is coprime to $n$), then some simples $[a,\ell]$ will have a twist $\theta([a,\ell])$ of order exactly $n^2$, rather than $n$.

 $\omega$ affects the $S$-matrix of $\mathbf{Mod}\,\,\cV^G$ too, as well as the modularity of the twisted twining characters of \textbf{TwMod}$_G\,\cV$. Let $M^g$ be the unique  simple $g$-twisted $\cV$-module, and define $Z_{g,h}(\tau)=\mathrm{Tr}_{M^g}q^{L_0-c/24}$ for any $h\in G$ with $M^{h^{-1}gh}\cong M^g$ (so in particular $h$ commutes with $g$). Then if $[\omega]=[1]$, these are permuted by the modular group SL$_2(\bbZ)$:  $$Z_{g,h}\left(\frac{a\tau+b}{c\tau+d}\right)=Z_{g^ah^c,g^bh^d}(\tau)$$
 But if $[\omega]\ne [1]$, then the right-side must be multiplied by some root of unity depending on $\omega$. For example, the modularity of these twisted twining characters for Monstrous Moonshine involve 24th roots of unity, so the gauge anomaly $[\omega]$ there   must be order at least 24. We now know \cite{JF} its order is exactly 24.
 
 The fusion rules of $\mathbf{Mod}\,\,\cV^G$ can change with $\omega$ (in contrast to those of $\mathbf{Vec}_G^\omega$ which are unchanged). For example, for $\Z/n$, the fusion ring of $\cZ(\mathrm{Vec}_{\Z/n}^\omega)$ is the group ring of $\Z/d\times\Z/(n^2/d)$, where $d$ is the gcd of $n$ with twice the order of $[\omega]$. For example, $\mathfrak{sl}(n^2)$ at level $k=1$ (with cyclic fusions $\Z[\Z/n^2]$) can be realized as a holomorphic orbifold by $\Z/n$, i.e.\ as $\cZ(\mathrm{Vec}_{\Z/n}^\omega)$
for some $\omega$ of maximal order.
 
 In fact the numbers of simples can decrease for $[\omega]\ne[1]$ compared to $[\omega]=[1]$. For example, though this won't happen for $G=\Z/n$ nor $(\Z/n)^2$, it can for  $(\Z/n)^3$. A simple example is $G=(\Z/2)^3$, which for an appropriate $\omega\in H^3((\Z/2)^3,\bbC^\times)\cong(\Z/2)^7$ recovers $\cZ(\mathbf{Vec}_{D_4}^\omega)$ for dihedral $D_4$, which has 22 simples (instead of the $(2^3)^2=64$ we would have expected).
 
 The numbers of VOA extensions of $\cV^G$ (equivalently, the \'etale algebras of $\cZ(\mathbf{Vec}_G^\omega)$)  typically decreases for nontrivial $[\omega]$. For example (see Table 1 in \cite{EG}),   $\cZ(\mathbf{Vec}_{\mathrm{Sym}(3)}^\omega)$ has exactly 8, 6, 5, 4 \'etale algebras when $[\omega]$ has order 1, 2, 3, 6 respectively in $H^3(\mathrm{Sym}(3),\bbC^\times)\cong\bbZ_6$. A simpler example: $(\cV_{E_8})^{\Z/2}$ has either 2 or 1 nontrivial VOA extensions, depending on which $\Z/2$ subgroup is chosen, as they have different $\omega$.

\section{Trivial module-maps and quantum cleft extensions}

Let $\cV$ now be any pointed VOA, and write \textbf{Mod}$\,\,\cV=\mathbf{Vec}_A^q$ for some abelian group $A$ and non-degenerate quadratic form $q$. Let $G$ be a finite group of automorphisms of $\cV$, and let $\rho:G\to \mathrm{O}(A,q)\le\mathrm{Aut}(A)$ be the corresponding module-map.
We will see this section that the arguments of Section 3 generalize naturally and completely to the situation where $\rho$ is the trivial map (i.e.\ $\rho(g)=\mathrm{Id}_A$ $\forall g\in G$). In this case we say  $G$ has a \textit{trivial module-map}. 
 The  Dijkgraaf-Witten story studied in Section 3 is the special case corresponding to $A=0$.

Trivial module-maps  are fairly common: e.g.\ any $\cV$-automorphism of prime order $\ge|A|$ must fix $A$ pointwise. This necessarily happens when $|A|\le 2$ (last section we studied the case where $|A|=1$). The $\rho$ trivial situation has been studied already in the literature. For example, Naidu \cite{Naidu} and Mason--Ng \cite{MN}   introduced finite-dimensional quasi-Hopf algebras. We prove these are essentially equivalent, and their representations recover the MTC \textbf{Mod}$\,\,\cV^G$ in this trivial module-map situation.

\subsection{The relevant tensor categories}

\begin{theorem} \label{stable} Suppose $\mathbf{D}$ is a braided $G$-crossed extension of a MTC $\mathbf{Vec}_A^q=\mathbf{Vec}_A^{(\omega,c)}$, with trivial module-map $\rho:G\to\mathrm{O}(A,q)$. Then $\mathbf{D}=\mathbf{Vec}_\Gamma^{\tilde{\omega}}$ for some central extension $\Gamma$ of $G$ by $A$, and $[\widetilde{\omega}|_A]=[\omega]$ in $H^3(A,\bbC^\times)$. Moreover, $(\mathbf{Vec}_A^q)^{\mathrm{rev}}$ is braided tensor equivalent to a full subcategory in $Z(\mathbf{Vec}_\Gamma^{\tilde{\omega}})$, and the equivariantization $\mathbf{D}^G$  is braided tensor equivalent to the M\"uger  centralizer of $ (\mathbf{Vec}_A^q)^{\mathrm{rev}}$ in $\cZ(\mathbf{Vec}_\Gamma^{\tilde{\omega}})$: $Z(\mathbf{Vec}_\Gamma^{\tilde{\omega}})\cong(\mathbf{Vec}_A^q)^{\mathrm{rev}}\boxtimes \mathbf{D}^G$.
\end{theorem}

\begin{proof} Choose any $g\in G$. By  Lemma \ref{rankDg}, for any $g\ne e$ in $G$, the  number of $g$-graded simples in \textbf{D} equals $|A|$, which also equals FPdim$(\mathbf{D}_g)$. Since these are equal,   the Frobenius-Perron dimension of every simple in $\mathbf{D}_g$ must be 1, i.e.\ every simple in $\mathbf{D}_g$ must be invertible. Thus $\mathbf{D}$ is a pointed fusion category, so equals $\mathbf{Vec}_\Gamma^{ \tilde{\omega}}$ for some finite group $\Gamma$ and some 3-cocycle $\widetilde{\omega}\in Z^3(\Gamma,\bbC^\times)$. For each $\gamma\in\Gamma$ let $X^\gamma$ denote the corresponding simple object of $\mathbf{D}$.

Let $\partial$ give the grading $\partial(X^\gamma)\in G$. We can view this as a function $\Gamma\to G$. Then $\partial:\Gamma\to G$ is a group homomorphism, since $\mathbf{D}_g\otimes\mathbf{D}_{g'}$ lies in $\mathbf{D}_{gg'}$. Also, $\mathbf{D}_e=\mathbf{Vec}_A^\omega$, so the kernel of $\partial$ equals $A$. Thus $\Gamma$ is an (abelian) extension of $G$ by $A$, as in \eqref{gpext}. Since $G$ acts trivially on $A$, $A$ must lie in the centre of $\Gamma$. 

Because $\mathbf{D}_e=\mathbf{Vec}_A^{(\omega,c)}$, $\left[\widetilde{\omega}|_A\right]=[\omega]$ in $H^3(A,\bbC^\times)$. The rest of the theorem follows from Proposition \ref{Z(CG)}. \end{proof}

 A fusion category $\mathbf{C}$ is called \textit{group-theoretical} if it is Morita equivalent to a pointed fusion category, i.e.\ if its centre $\cZ(\mathbf{C})$ is braided tensor equivalent to the centre of a pointed fusion category.  By Theorem 7.2 of \cite{NNW}, any equivariantization $\mathbf{D}^G$ arising in Theorem \ref{stable} is group-theoretical, and conversely, any group-theoretical MTC arises as $\mathbf{D}^G$ in Theorem \ref{stable}.
 
The braided $G$-crossed category structure on \textbf{Vec}$_\Gamma^{\tilde{\omega}}$  is given explicitly in Section 4 of \cite{Naidu}.
In particular, the $G$-grading on  \textbf{Vec}$^{\tilde{\omega}}_\Gamma$ is given by $\partial(X^\gamma)=\gamma A\in\Gamma/A$. The $G$-action is given by ${}^g(X^\gamma)=X^{\tilde{g}\gamma \tilde{g}^{-1}}$ for any lift $\tilde{g}$ of $g\in \Gamma/A$ to $\Gamma$. Because $A$ is central, this action is well-defined.

\begin{corollary}  Let $\cV$ be a pointed VOA with $\mathbf{Mod}\,\,\cV=\mathbf{Vec}_A^q$. Let $G$ be any finite group of $\cV$-automorphisms with trivial module-map. Then $\mathbf{TwMod}_G\,\cV=\mathbf{Vec}_\Gamma^{\tilde{\omega}}$ for some  $\Gamma$ and $\widetilde{\omega}$ as in Theorem \ref{stable}. Moreover, $\mathbf{Mod}\,\,\cV^G$   is braided tensor equivalent to the equivariantization $\mathbf{D}^G$. \end{corollary}

The converse of the theorem is also true, thanks to Lemma 5.1 of \cite{MuA}:
\begin{proposition} \label{conv} Suppose $\mathbf{D}$ is a pointed braided $G$-crossed extension of a MTC $\mathbf{Vec}_A^q$. Then the corresponding module-map $\rho$ is trivial.\end{proposition}

Even though $H^3(G,A)$ may be large, the obstruction $o_3(\rho)$ (which equals the Eilenberg-Mac Lane obstruction to \eqref{gpext})  always must vanish for the trivial module-map $\rho$ since $\Gamma=A\times G$ works. The $H^2$ torsor is clear: it corresponds to the different central extensions $\Gamma$ of $G$ by $A$. The obstruction $o_4$ is much more delicate (see Corollary \ref{o4} and Proposition \ref{o4fails} below); when it vanishes, the $H^3$-torsor involves taking any $\alpha\in Z^3 (G,\bbC^\times)$, inflating it to $\alpha'\in Z^3(\Gamma,\bbC^\times)$ in the usual way, and multiplying  $\widetilde{\omega}$ by $\alpha'$. We still have $\widetilde{\omega}\alpha'|_A=\omega$ thanks to inflation.

\subsection{VOA reconstruction}

For this subsection, assume Conjecture \ref{conj} (though this is unnecessary when $G$ is solvable). Use of the conjecture can also be avoided by using conformal nets of factors on $S^1$ rather than VOAs. Our proof of VOA reconstruction here loosely follows the proof of Theorem \ref{reconstruction_theorem}.

\begin{theorem} \label{recon}  Suppose $\mathbf{Vec}_\Gamma^{\tilde{\omega}}$ is a braided $G$-crossed extension of a MTC $\mathbf{Vec}^{(\omega,c)}_A$, for a central extension $\Gamma$ of $G$ by $A$, as in Theorem \ref{stable}. Let $\mathbf{D}$ be the $G$-equivariantization of $\mathbf{Vec}_\Gamma^{\tilde{\omega}}$. 
Then there exists a strongly rational  VOA $\cV$ with $\mathbf{Mod}\,\,\cV\cong\mathbf{Vec}_A^{(\omega,c)}$ which has a group of VOA automorphisms  isomorphic to $G$, and $\mathbf{TwMod}_G\,\cV\cong\mathbf{Vec}_\Gamma^{\tilde{\omega}}$ and $\mathbf{Mod}\,\,\cV^G\cong\mathbf{D}$.
\end{theorem}

\begin{proof} Follow the notation in the statement of Theorem \ref{stable}. By Theorem \ref{doubleVOA}, there is a holomorphic VOA $\cV_{\mathrm{hol}}$ and a group of $\cV_{\mathrm{hol}}$-automorphisms isomorphic to $\Gamma$ such that $\mathbf{Mod}\,\,\cV_{\mathrm{hol}}^\Gamma\cong\cZ(\mathbf{Vec}_\Gamma^{\tilde{\omega}})$ and $\mathbf{TwMod}_\Gamma\,\cV_{\mathrm{hol}}\cong\mathbf{Vec}_\Gamma^{\tilde{\omega}}$. By Theorem 2 of \cite{EGty}, there is a lattice VOA $\cV_L$  with $\mathbf{Mod}\,\,\cV_L\cong\mathbf{Vec}^\omega_A$. Choose any \'etale algebra $B_A$ in $\cZ(\mathbf{Vec}_A^\omega)\cong\mathbf{Vec}_A^\omega\boxtimes(\mathbf{Vec}_A^\omega)^{\mathrm{rev}}$ such that $(\mathbf{Rep}_{\cZ(\mathbf{Vec}_A^\omega)}\,B_A)^{\mathrm{loc}}\cong \mathbf{Vec}$ -- e.g.\ $B_A=\oplus_{\chi\in\hat{A}}[0,\chi]$ works.

Note that the orbifold $(\cV_L\otimes\cV_{\mathrm{hol}})^{1\times A}$ (regarding $A$ as a normal subgroup of $\Gamma$) is a strongly rational VOA with $\mathbf{Mod}\,\,(\cV_L\otimes\cV_{\mathrm{hol}})^{1\times A}\cong \mathbf{Vec}_A^{(\omega,c)}\boxtimes\cZ(\mathbf{Vec}_A^\omega)\cong \cZ(\mathbf{Vec}_A^\omega)\boxtimes \mathbf{Vec}_A^{(\omega,c)}$. Let $\cV$ denote the extension of $(\cV_L\otimes\cV_{\mathrm{hol}})^{1\times A}$ by  $B_A\boxtimes 1$. Then $\cV$ is strongly rational, with $\mathbf{Mod}\,\,\cV\cong\mathbf{Vec}_A^{(\omega,c)}$, and $G=\Gamma/A$ acts on $\cV_{\mathrm{hol}}^A$ hence $\cV$ as VOA automorphisms.

Perhaps the easiest way to see the latter is to take first the full orbifold $(\cV_L\otimes\cV_{\mathrm{hol}})^{1\times \Gamma}$, which has MTC  $\mathbf{Vec}_A^{(\omega,c)}\boxtimes\cZ(\mathbf{Vec}_\Gamma^{\tilde{\omega}})\cong \cZ(\mathbf{Vec}_A^\omega)\boxtimes \mathbf{D}$, since $\cZ(\mathbf{Vec}_\Gamma^{\tilde{\omega}})\cong(\mathbf{Vec}_A^{(\omega,c)})^{\mathrm{rev}}\boxtimes\mathbf{D}$  by Proposition \ref{Z(CG)}. Since by hypothesis $\mathbf{D}$ is the $G$-equivariantization of the braided $G$-crossed  extension  $\mathbf{Vec}_\Gamma^{\tilde{\omega}}$ of $\mathbf{Vec}_A^{(\omega,c)}$, $\mathbf{D}$ contains an \'etale algebra $B_G$ (a copy of the algebra of functions on $G$) such that the de-equivariantization $\mathbf{D}_G$ is $\mathbf{Vec}_\Gamma^{\tilde{\omega}}$. Note that the two \'etale algebras $B_A\boxtimes 1$ and $1\boxtimes B_G$ commute. Extending  $(\cV_L\otimes\cV_{\mathrm{hol}})^{1\times A}$ by both returns us to $\cV$. The $G$-action on $\cV$ comes from the extension by $1\boxtimes B_G$.

In any case we see that the orbifold $\cV^G$ obtains $(\cV_L\otimes\cV_{\mathrm{hol}})^{1\times \Gamma}$ extended by $B_A\boxtimes 1$, which has $\mathbf{Mod}\,\,\cV^G\cong\mathbf{D}$. Also, $\mathbf{TwMod}_G\,\cV$ recovers $\mathbf{Vec}_\Gamma^{\tilde{\omega}}$, by construction of $\cV$.
\end{proof}




Using  Theorem 7.2 of \cite{NNW} together with Theorem \ref{recon}, the following is immediate:

\begin{corollary}\label{recon_gpth} Let $\mathbf{D}$ be a group-theoretical MTC. Then there exists a pointed VOA $\cV$ and a finite group $G$ of $\cV$-automorphisms such that $\mathbf{Mod}\,\,\cV^G$ is braided tensor equivalent to $\mathbf{D}$.\end{corollary}

This result is more general than it may first look. For example, all semisimple quasi-Hopf
algebras of dimension a prime power,  have  categories of representations which are group-theoretical (see Remark 1.6 of \cite{drinfeld}).

\subsection{Quasi-Hopf interpretations}

Since all FPdims in the equivariantization $(\mathbf{Vec}_\Gamma^{\tilde{\omega}})^G$ are integers (being as it is a subcategory of the quantum double $\cZ(\mathbf{Vec}_\Gamma^{\tilde{\omega}})$), that MTC will be the category of representations of a quasitriangular quasi-Hopf algebra. We shall describe and relate two constructions of such quasi-Hopf algebras.

Given a central extension $\Gamma$ of $G$ by $A$ as in \eqref{gpext}, Naidu \cite{Naidu} determines the data needed to define a braided $G$-crossed category structure on some $\mathbf{Vec}^{\tilde{\omega}}_\Gamma$  (when such a structure exists). He calls this data a \textit{quasi-abelian 3-cocycle} (see Definition \ref{qa3c_definition} in our Appendix), so named because it is an extension of Eilenberg-Mac Lane's notion of abelian 3-cocycle.
Given a quasi-abelian 3-cocycle for such a $\Gamma$, Naidu defines a quasi-Hopf algebra structure on $\bbC_{\tilde{\omega}}^\Gamma\otimes\bbC G$ (see Definition \ref{naiqH} in our Appendix) and proves that its category of finite-dimensional representations is  braided tensor equivalent to the $G$-equivariantization of the corresponding  $G$-crossed category  $\mathbf{Vec}^{\tilde{\omega}}_\Gamma$.

Mason--Ng \cite{MN} proposed a similar quasi-Hopf algebra. Let $\mathbf{Vec}_A^{(\omega,c)}$ be a MTC, $\Gamma$ be a central extension of a finite group $G$ by $A$, and $\widetilde{\omega}\in Z^3(\Gamma,\bbC^\times)$. Suppose in addition that 

\smallskip\noindent{\textbf{(MN1)}} $ \theta_a:\Gamma\times \Gamma\to\bbC^\times$ is coboundary on $\Gamma\times\Gamma$ for all $a\in A$, where $\theta_a(g,h)$ is defined in \eqref{thetaform};

\smallskip\noindent{\textbf{(MN2)}} Assuming (MN1), write $\theta_a(g,h)=\frac{t_a(g)\,t_a(h)}{t_a(gh)}$ for some normalized 1-cochains $t_a:\Gamma\to\bbC^\times$ ($t_0=1$). Then $\beta:A\times A\to\widehat{\Gamma}$ is also coboundary, where $$\beta(a,b)(g)=\frac{t_a(g)\,t_b(g)}{t_{a+b}(g)}\theta_g(a,b)\,;$$

\noindent{\textbf{(MN3)}} Assuming (MN2), a normalized 1-cochain $\nu:A\to\widehat{\Gamma}$ can be found so that $\beta(a,b)=\frac{\nu(a)\,\nu(b)}{\nu(a+b)}$ and $c(a,b)=\frac{\nu(a)(b)}{t_b(a)}$, $\forall a,b\in A$.

\smallskip $c$ in (MN3) is the braiding. Choosing choosing different 1-chains $t_a$ satisfying $\delta t_a=\theta_a$  gives an equivalent braiding, but a different 1-cochain $\nu$ with $\delta\nu=\beta$ can affect the braiding $c$ nontrivially.

When those  conditions are both satisfied, we get a quasi-Hopf algebra $D^{\tilde{\omega}}(\Gamma,A)$, namely a \textit{cleft extension} of the group algebra $\bbC \Gamma$ by the twisted dual group algebra $\bbC^\Gamma_{\tilde{\omega}}$. As with Naidu's quasi-Hopf algebras, the underlying vector space is $\bbC_{\tilde{\omega}}^\Gamma\otimes\bbC G$. The twisted quantum double $D^{\tilde{\omega}}(\Gamma)$ is recovered by the choice $A=0$. Moreover, the category  $\mathbf{Rep}\,\,D^{\tilde{\omega}}(\Gamma,A)$  of finite-dimensional representations of $D^{\tilde{\omega}}(\Gamma,A)$ is a MTC.  They conjecture:

\begin{conjecture} \label{MasNg}\cite{MN2} Suppose $\cV$ is a pointed VOA with $\mathbf{Mod}\,\,\cV\cong\mathbf{Vec}_A^{(\omega,c)}$. Let $G$ be a finite group of automorphisms of $\cV$ with trivial module-map. Then  $\mathbf{Mod}\,\,\cV^G\cong \mathbf{Rep}\,\,D^{\tilde{\omega}}(\Gamma,A)$ for some central extension $\Gamma$ of $G$ by $A$, and some $\widetilde{\omega}\in Z^3(\Gamma,\bbC^\times)$ for which properties (MN1)-(MN3) are satisfied. Conversely, if properties (MN1)-(MN3) are satisfied,  then there is a pointed VOA $\cV$ and a group $G\cong\Gamma/A$ of automorphisms of $\cV$ such that  $\mathbf{Mod}\,\,\cV^G\cong \mathbf{Rep}\,\,D^{\tilde{\omega}}(\Gamma,A)$.\end{conjecture}

The relation between the Naidu and Mason--Ng quasi-Hopf algebras is established in the Appendix (see Lemma \ref{qhequiv}). This then proves their Conjecture  in the affirmative. It also gives us: 

\begin{corollary} \label{o4} Suppose $\mathbf{Vec}_A^{(\omega,c)}$ is a MTC, $\Gamma$ is a central extension of a finite group $G$ by $A$, and $\widetilde{\omega}\in Z^3(\Gamma,\bbC^\times)$ has $\widetilde{\omega}|_A=\omega$.  Then the ENO obstruction $o_4$ vanishes and $\mathbf{Vec}_\Gamma^{\tilde{\omega}}$ has braided $G$-crossed structure iff (MN1)-(MN3) are satisfied.\end{corollary}

Recall that the other obstruction, $o_3$, automatically vanishes in this trivial module-map setting. We will give applications of this corollary in Section 4.5. Certainly there are many examples where the first sentence of Corollary \ref{o4} is satisfied, yet $\mathbf{Vec}_\Gamma^{\tilde{\omega}}$ is not a braided $G$-crossed extension of $\mathbf{Vec}_A^{(\omega,c)}$. For example take $A$ to be odd order (so $[\omega]=[1]$) and $\tilde{\omega}=1$. Then $t_a=1=\nu$ satisfies $\delta t_a=\theta_a$ and $\delta\nu=\beta$, in which case $(\mathbf{Vec}_\Gamma^{\widetilde{\omega}})_e=\mathbf{Vec}_A\not\cong\mathbf{Vec}_A^{(\omega,c)}$  (the condition on $c$ in (MN3) is missing).  If in addition $\Gamma$ is perfect (i.e.\ $\Gamma=[\Gamma,\Gamma]$), no other choice of $t_a,\nu$ will work either. 

See also the discussion around Proposition \ref{o4fails}.

\subsection{The simple objects and modular data of $D^{\omega}(K,A)$}
\label{simple_obj_cleft_section}

Let $\mathbf{D}=\mathbf{Mod}\,\,\cV^G$ where $\mathbf{Mod}\,\,\cV\cong\mathbf{Vec}^{(\omega,c)}_A$ when the module-map is trivial, or equivalently $\mathbf{D}$ is the equivariantization $(\mathbf{Vec}_\Gamma^{\tilde{\omega}})^G$, or equivalently $\mathbf{D}=\mathbf{Rep}\,\, D^{\tilde{\omega}}(\Gamma,A)$. Since $\mathbf{D}$ is a braided fusion subcategory of $\cZ(\mathbf{Vec}_\Gamma^{\tilde{\omega}})$, it suffices to identify which simples of $\cZ(\mathbf{Vec}_\Gamma^{\tilde{\omega}})$ lie in $\mathbf{D}$. Since $\cZ(\mathbf{Vec}_\Gamma^{\tilde{\omega}})\cong\mathbf{D}\boxtimes (\mathbf{Vec}_A^{(\omega,c)})^{\mathrm{rev}}$, the $S$ matrix entries for $\mathbf{D}$ equal the corresponding $S$ matrix entries for $\cZ(\mathbf{Vec}_\Gamma^{\tilde{\omega}})$, rescaled by $\sqrt{|A|}$ (the value of $1/S_{0,0}$ for $(\mathbf{Vec}_A^{(\omega,c)})^{\mathrm{rev}}$). The  $S$ matrix for the centre $\cZ(\mathbf{Vec}_\Gamma^{\tilde{\omega}})$ is given in \cite{CGR} for any twist and group.

Recall from Section 3.1 that the (equivalence classes of)  simple objects of $\mathcal{Z}(\mathbf{Vec}_\Gamma^{\tilde{\omega}})$ are parametrized by $[\gamma,\chi]$ where $\gamma$ is a representative of a conjugacy class in $\Gamma$ and $\chi$ is the character of an irreducible projective representation of the centralizer $C_\Gamma(\gamma)$ with 2-cocycle $\theta_\gamma$ determined by $\widetilde{\omega}$ as in \eqref{thetaform}. 

Then, for any $[\gamma,\chi]\in\mathbf{D}\subset \mathcal{Z}(\mathbf{Vec}_\Gamma^{\tilde{\omega}})$, the ribbon twist is $\chi(\gamma)/\chi(e)$ and  FPdim($[\gamma,\chi])=\|K_\gamma\|\,\chi(e)$, both inherited from $\mathcal{Z}(\mathbf{Vec}_\Gamma^{\tilde{\omega}})$, where $K_\gamma$ is the conjugacy class of $\gamma$ in $\Gamma$.

There are different ways to identify the simple objects of $\mathbf{D}$. For example, one could use  the description of $\mathrm{\textbf{Rep}}\,\, H(\omega,\gamma,\mu)$ given in \cite[Section 5]{Naidu}. Alternatively one may use \cite[Corollary 2.13, Theorem 3.9]{BN} to obtain the simple objects and fusion rules. The approach we take in this subsection is to first identify the subcategory $(\mathbf{Vec}_A^{(\omega,c)})^{\mathrm{rev}}$ and then take its M\"uger centralizer. In the last subsection of the appendix,  we take a more quasi-Hopf point of view and construct the representations of $D^{\tilde{\omega}}(\Gamma,A)$ directly, following \cite{DPR}.

 
 Recall the 1-cochains $t_a$ and $\nu$ from (MN2),(MN3).
 
\begin{proposition} For each $a\in A$ define  $x_a:=t_a/\nu(a)$. Then $[a,x_a]$ is a simple object in $\cZ(\mathbf{Vec}_\Gamma^{\tilde{\omega}})$. The subcategory they generate is equivalent as a MTC to $(\mathbf{Vec}_A^{(\omega,c)})^{\mathrm{rev}}$. The equivariantization $\mathbf{D}=(\mathbf{Vec}_\Gamma^{\tilde{\omega}})^G$ consists of the simple objects $[\gamma,\chi]\in\mathbf{Vec}_\Gamma^{\tilde{\omega}}$ satisfying $\chi(a)=\overline{x_a(\gamma)}\,\mathrm{dim}\,\chi$ for all $a\in A$.\end{proposition}

\begin{proof}  First note that for any $a\in A$, $x_a:=t_a/\nu(a)$ is a 1-dimensional projective representation of $\Gamma$ with multiplier $$\frac{t_a(b)}{\nu(a)(b)}\frac{t_a(b')}{\nu(a)(b')}\frac{\nu(a)(b+b')}{t_a(b+b')}=\theta_a(b,b')$$
so $[a,x_a]\in\cZ(\mathbf{Vec}_\Gamma^{\tilde{\omega}})$ -- in fact it's a simple current. 

Recall that a pointed MTC is uniquely determined by its fusions (defining an abelian group) and the ribbon twists $\theta(x)$, determining the quadratic form $q(x)=c(x,x)$. We can access the fusions through the $S$ matrices.
By e.g.\ (5.30) of \cite{MN}, we compute $$S^\Gamma_{[a,x_a],[b,x_b]}=\frac{1}{|\Gamma|}\frac{t_a(b)}{\nu(a)(b)}\frac{t_b(a)}{\nu(b)(a)}=\frac{1}{|\Gamma|}\overline{c(a,b)\,c(b,a)}=\frac{\sqrt{|A|}}{|\Gamma|}\overline{S^A_{a,b}}$$
where $S^\Gamma$ is the (normalized) $S$ matrix of $\cZ(\mathbf{Vec}_\Gamma^{\tilde{\omega}})$ and  $S^A$ is that of $\mathbf{Vec}_A^{({\omega},c)}$. By Verlinde \eqref{verl}, for each $b\in A$ the maps $a\mapsto S^A_{a,b}/S^A_{0,b}$ resp.\ $[\gamma,\chi]\mapsto S^\Gamma_{[\gamma,\chi],[b,x_b]}/S^\Gamma_{[e,1],[b,x_b]}$  define  1-dimensional representations of the fusion rings of $\mathbf{Vec}_A^{(\omega,c)}$ resp.\ $\cZ(\mathbf{Vec}_\Gamma^{\tilde{\omega}})$. Hence $$\frac{S^\Gamma_{[a,x_a],[b,x_b]}}{S^\Gamma_{[0,1],[b,x_b]}}\,\frac{S^\Gamma_{[a',x_{a'}],[b,x_b]}}{S^\Gamma_{[0,1],[b,x_b]}}=\overline{\frac{S^A_{a,b}}{S^A_{0,b}}}\,\overline{\frac{S^A_{a',b}}{S^A_{0,b}}}
=\overline{\frac{S^A_{a+a',b}}{S^A_{0,b}}}=\frac{S^\Gamma_{[a+a',x_{a+a'}],[b,x_b]}}{S^\Gamma_{[0,1],[b,x_b]}}$$
Nondegeneracy of $\mathbf{Vec}_A^{(\omega,c)}$ (=invertibility of $S^A$) implies the matrix $ {\frac{S^\Gamma_{[a,x_a],[b,x_b]}}{S^\Gamma_{[0,1],[b,x_b]}}} $ is also invertible, so the fusions of the simple currents $[a,x_a]$ match those of $(\mathbf{Vec}_A^{(\omega,c)})^{\mathrm{rev}}$. 

Finally, using (MN3) we find that the ribbon twists $\theta([a,x_a])=t_a(a)/\nu(a)(a)$ and $\theta(a)=c_{a,a}$ are complex conjugates. Therefore the $[a,x_a]$ generate a fusion subcategory in $\cZ(\mathbf{Vec}_\Gamma^{\tilde{\omega}})$ equivalent as a MTC to $(\mathbf{Vec}_A^{(\omega,c)})^{\mathrm{rev}}$.

Theorem 1.2 of  \cite{NNW} classifies all fusion subcategories of $\cZ(\mathbf{Vec}_\Gamma^{\tilde{\omega}})$ in terms of a triple $(K,H,B)$ where $K,H$ are normal subgroups of $\Gamma$ and $B:K\times H\to\bbC^\times$ is a $\Gamma$-invariant $\widetilde{\omega}$-bicharacter. For the subcategory generated by the $[a,x_a]$, the triple is $(A,\Gamma,B)$ where $B(a,\gamma)=x_a(\gamma)$ $\forall a\in A,\gamma\in\Gamma$. Lemma 5.10 of \cite{NNW} tells us that the M\"uger centralizer of $(\mathbf{Vec}_A^{(\omega,c)})^{\mathrm{rev}}$ has triple $(\Gamma,A,\overline{B^{\mathrm{op}}})$ where $\overline{B^{\mathrm{op}}}(\gamma,a)=\overline{x_a(\gamma)}$. Since the M\"uger centralizer will be $\mathbf{D}$, this translates to the condition on $[\gamma,\chi]$ given in the Proposition.
\end{proof}

Being an equivariantization, $\mathbf{D}$ (which once again is \textbf{Mod}$\,\,\cV^G$ in the VOA setting) must contain an \'etale algebra $B_G$ which is a copy of the regular representation of $G$. Indeed, $B_G=\oplus_{\chi}\mathrm{dim}\,\chi\,[e,\chi]\in\mathbf{D}$ where $\chi$ runs over all ordinary irreducible characters of $\Gamma $ satisfying $\chi(a)=1$ $\forall a\in A$, or equivalently $\chi$ defines an ordinary irreducible character of $G=\Gamma/A$.

Let $M^\gamma$ be the unique simple $\gamma$-twisted $\cV$-module, or equivalently the simple $v_\gamma\in\mathbf{Vec}_\Gamma^{\tilde{\omega}}$. Restriction of $M^\gamma$ to $\cV^G$ is $\oplus_{\chi} \mathrm{dim}\,\chi\, [\gamma,\chi]$ where the sum is over all $\chi$ such that $[\gamma,\chi]\in\mathbf{D}$, i.e.\ all irreducible projective characters $\chi$ of the centralizer $C_\Gamma(\gamma)$ with 2-cocycle $\theta_\gamma$ such that $\chi(a)=\overline{x_a(\gamma)}\,\mathrm{dim}\,\chi$ $\forall a\in A$. Induction sends $[\gamma,\chi]\in\mathbf{D}$ to $\oplus_{h}\mathrm{dim}\,\chi\,M^h$ where the sum is over the conjugacy class of $\gamma$.

\subsection{Examples}

As mentioned earlier, one situation where the module-map is guaranteed to be trivial, is when $A=\Z/2$. In this case, recall that $\mathbf{Vec}_{\Z/2}^{(\omega,c)}$ can be a MTC only when $[\omega]$ is nontrivial in $H^3(\Z/2,\bbC^\times)$. Up to equivalence, there are precisely two MTC for $A=\Z/2$, corresponding to either sign in the braiding $c(1,1)=\pm\mathrm{i}$. The lattices VOAs $\cV_{A_1}$ and $\cV_{E_7}$ realize these MTC for $c(1,1)=\pm\mathrm{i}$ respectively.

Consider $\cV_{A_1}$, also known as the level 1 affine $\mathfrak{sl}_2$ VOA $\cV(\mathfrak{sl}_2,1)$. Its automorphism group  is SO$_3(\mathbb{R})$, which acts by the adjoint representation on the homogeneous subspace $ (\cV_{A_1})_1\cong\bbC^3$, which generates  $\cV_{A_1}$. Its $\Z/2$-central extension SU(2) acts as automorphisms on the nontrivial $\cV_{A_1}$-module, which is generated by its lowest $L_0$-eigenspace, a copy of the irreducible 2-dimensional SU(2)-module.  As is well-known, the finite subgroups $G$ of automorphisms of $\cV_{A_1}$ therefore fall into the familiar ADE pattern: cyclic, dihedral, and Alt$_4,\mathrm{Sym}_5,\mathrm{Alt}_5$, and the relevant central extensions $\Gamma$ are their lifts into SU(2).  Then this falls into the framework of this section. (All such $G$ are solvable, except Alt$_5$, so for the latter we require Conjecture 1 to guarantee strong rationality of the VOA orbifold.)

More generally, the Main Theorem of  \cite{MN3} concerns groups $\Gamma$ which contain exactly one subgroup $A$ of order 2 (which therefore lies in the centre). Examples of such groups are  SL$_2(\mathbb{F})$ for $\mathbb{F}$ a finite field of odd order,  the generalized quaternion groups, and the finite subgroups of SU(2). The full classification of such $\Gamma$ is given in Section 3.4 of \cite{MN3}.

Using our Corollary \ref{o4}, their Main Theorem says in our language:

\begin{corollary} \label{mn3} Let $\Gamma$ be a central extension of a finite group $G$ by $A=\Z/2$, and assume $A$ is the only order-2 subgroup of  $\Gamma$. Let  $\widetilde{\omega}\in Z^3(\Gamma,\bbC^\times)$ be such that $[\widetilde{\omega}|_{A}]$ is nontrivial. Then (MN1)-(MN2) are satisfied, and (MN3) defines the sign on a nondegenerate braiding $c$, such that $\mathbf{Vec}_\Gamma^{\tilde{\omega}}$ will be a braided $G$-crossed extension of the MTC $\mathbf{Vec}_A^{(\tilde{\omega}|_A,\pm\mathrm{i})}$ for that  sign.\end{corollary} 

 As explained in \cite{MN3}, the 2-torsion subgroup of $H^3(\Gamma,\bbC^\times)$ for such a group will be cyclic with the same order as any 2-Sylow subgroup of $\Gamma$, and the restriction $H^3(\Gamma,\bbC^\times)\to H^3(A,\bbC^\times)\cong\Z/2$  is nontrivial. So the condition in the Corollary that $[\widetilde{\omega}|_A]$ be nontrivial is equivalent to requiring that the class  $[\widetilde{\omega}]$ has full 2-order in $H^3(\Gamma,\bbC^\times)$.

What  does Corollary \ref{mn3} say about the mysterious obstruction $o_4$? For $A=\Z/2$ with trivial module-map, $o_4$ depends on both the sign of $c$  (which determines $\mathbf{Vec}_A^{(\omega,c)}$) and on $\Gamma$ (which tells us where we are on the $H^2$-torsor). Write $o_4(\Gamma,\pm)$ to emphasize that dependence. Choose any $\widetilde{\omega}$ with $[\widetilde{\omega}|_A]$ nontrivial. Then Corollary \ref{mn3} says $[o_4(\Gamma,\pm)]$ must vanish for the sign coming from (MN3). However, choose some integer $k\equiv -1$ (mod 4) and coprime to $|\Gamma|$. Then choosing $\widetilde{\omega}^k$ (as well as $t_a^k$ and $\nu^k$) gives the other sign in (MN3). So  Corollary \ref{mn3} implies $[o_4(\Gamma,\pm)]$ vanishes for both signs.

We can generalize much of this argument. Consider $A=\Z/p^n$ for any odd prime power. Then again there are up to MTC equivalence exactly two  MTC $\mathbf{Vec}_A^{(\omega,c)}$, determined by the value $c(1,1)=e^{2\pi \mathrm{i} \ell/p^n}$: whether or not $\ell$ is a quadratic residue mod $p$. For $A=\Z/2^n$, there are either two ($n=1$) or four ($n>1$) inequivalent MTC $\mathbf{Vec}_A^{(\omega,c)}$: write $c(1,1)=e^{2\pi \mathrm{i}\ell/2^{n+1}}$, then what matters is the value of $\ell$ (mod 4) respectively (mod 8). Again write $o_4(\Gamma,\ell)$ to cover these cases, where $\ell$ always must be coprime to $|A|$ (for nondegeneracy of the braiding).

\begin{proposition} \label{o4fails} Let $\Gamma$ be a central extension of $G$ by a cyclic group $A$ and let $\ell$ be any integer coprime to $|A|$. Then either $[o_4(\Gamma,\ell)]$ is trivial for all $\ell$, or nontrivial for all $\ell$.\end{proposition} 

\begin{proof} Suppose $[o_4(\Gamma,\ell)]$ vanishes for some $\ell$ coprime to $|A|$, and write $\mathbf{Vec}_A^{(\omega,c_\ell)}$ for the corresonding MTC. Then by Corollary \ref{o4}, there exists an $\widetilde{\omega}$ for $\Gamma$ with $[\widetilde{\omega}|_A]=[\omega]$, which passes (MN1)-(MN3) for some choice of $t_a,\nu$ and $c_\ell$. Now choose any other $\ell'$ coprime to $|A|$ and write $k\equiv \ell'/\ell$ (mod $|A|$). Then $\widetilde{\omega}^k,t_a^k,\nu^k$ satisfy (MN1)-(MN3) for $c_{\ell'}$ (and also $\omega^k,c^k$ satisfy the constraints (8.11) of \cite{book}). Thus   $[o_4(\Gamma,\ell')]$ also vanishes.\end{proof}

The consequences of some $[o_4(\Gamma,\ell)]$ not vanishing is significant for VOAs. Suppose for concreteness this happened for $0\to\Z/2\to\mathrm{SL}_2(\Z/5)\to \mathrm{ Alt}_5\to 1$ for $\ell=-1$ (Proposition \ref{o4fails} says it doesn't). That would imply that any Alt$_5$ orbifold of $\cV_{E_7}$ would be trivial, in the sense that \textbf{Mod}$\,\,(\cV_{E_7})^{\mathrm{Alt_5}}\cong\mathbf{Mod}\,\,\cV_{E_7}\boxtimes\cZ(\mathbf{Vec}_{\mathrm{Alt}_5}^{\tilde{\omega}})$ for some $\widetilde{\omega}\in Z^3(\mathrm{Alt}_5,\bbC^\times)$, i.e.\ that orbifold would involve the trivial central extension $\Gamma=\Z/2\times \mathrm{Alt}_5$ and not the binary icosahedral SL$_2(\Z/5)$.

Incidentally, the largest exceptional finite subgroup of PSU(3) is Alt$_6$. This acts on the lattice VOA $\cV_{A_2}$, and corresponds to $\Gamma$ being the Valentiner group, which is the unique nontrivial  extension of Alt$_6$  by $\Z/3$. This plays the role of the icosahedral example for $\cV_{A_1}$.

Incidentally, \cite{MN3} give two explicit conjectures in this context, concerning orbifolds of $\cV_{A_1}$ and $\cV_{E_7}$. These are special cases and refinements  of Conjecture 2 stated earlier, and so are proved by our results (provided in the nonsolvable cases you assume Conjecture 1).

The trivial module-map case was also encountered in \cite{Lam}. More precisely, the Main Theorem there stated that if $G$ is a cyclic group of isometries of an even positive definite lattice $L$, and $\widehat{G}$ is its lift to automorphisms of $\cV_L$, then $(\cV_L)^{\hat{G}}$ is pointed iff the corresponding module-map is trivial. Lam used this to prove a conjecture by H\"ohn concerning the holomorphic VOAs of central charge 24. The proof of this in \cite{Lam} depended explicitly on the existence of $L$. Using our methods, we can generalize considerably one direction, whereas the other direction is more delicate:

\begin{proposition} \label{lam} Let $\cV$ be any pointed VOA and $G$ any group of automorphisms of $G$. Assume $G$ is solvable (or if $G$ is not solvable, assume Conjecture 1). If $\mathbf{Mod}\,\,\cV^G$ is pointed, then the module-map is trivial and $G$ is abelian. Conversely, if the module-map is trivial, then $\mathbf{Mod}\,\,\cV^G$ is pointed iff $\cZ(\mathbf{TwMod}_G\,\cV)$ is pointed. Even when $G$ is cyclic, $\mathbf{Mod}\,\,\cV^G$ may not be pointed. \end{proposition}

\begin{proof} Let $\mathbf{D}=\mathbf{TwMod}_G\,\cV$, and suppose $\mathbf{D}^G=\mathbf{Mod}\,\,\cV^G$ is pointed. Let $x$ be any simple object in $\mathbf{D}$ (so it will be a simple current). Since induction from $\mathbf{D}^G$ to $\mathbf{D}$ is a tensor functor, Ind$(x)$ will be simple and invertible in $\mathbf{D}$ (with inverse Ind$(x^*)$). By Frobenius reciprocity, any simple in $\mathbf{D}$ arises as a subobject in some Ind$(x)$. Therefore every simple in $\mathbf{D}$ is invertible. Thus by Proposition \ref{conv} the  module-map is trivial.

Write $\mathbf{Mod}\,\,\cV=\mathbf{Vec}_A^q$. Suppose the module-map is trivial. Then by Theorem \ref{stable}, $\mathbf{D}=\mathbf{Vec}_\Gamma^{\tilde{\omega}}$ for some $\Gamma$ and $\widetilde{\omega}$. Since  $\cZ(\mathbf{D})\cong(\mathbf{Vec}_A^q)^{\mathrm{rev}}\boxtimes\mathbf{D}^G$, we see that  $\cZ(\mathbf{D})$ is pointed iff $\mathbf{D}^G$ is. This requires  $\Gamma$ (hence $G$) to be abelian, but $\cZ(\mathbf{D})$ may not be pointed even if $\Gamma$ is abelian, as we'll see shortly.\end{proof} 

It is interesting that one direction of the Main Theorem of \cite{Lam} fails if we generalize \textit{lattice} VOA there to \textit{pointed} VOA. For example,   if $A=\Z/n\oplus\Z/n$ and $G=\Z/n$ (for any $n$), choose $\Gamma=A\times G=\Z/n\times\Z/n\times\Z/n$; then it is possible to choose 3-cocycles so that $\mathbf{D}^G=\mathbf{Mod}\,\,\cV^G$ is not pointed. In particular choose $\widetilde{\omega}\in Z^3((\Z/n)^3,\bbC^\times)$ to be $\widetilde{\omega}((a_1,a_2,a_3),(b_1,b_2,b_3),(c_1,c_2,c_3))=e^{2\pi\mathrm{i} a_1b_2c_3/n}$. This restricts to the trivial cocycle on $A$ (and for any $n$, it is possible to find a $c$ such that $\mathbf{Vec}_{(\Z/n)^3}^{(1,c)}$ is a MTC). As explained in Section 6.2 of \cite{CGR}, $\cZ(\mathbf{Vec}_{(\Z/n)^3}^{\tilde{\omega}})$ is not pointed.

Nevertheless, in the setting of the trivial module-map, it is very limiting to restrict to cyclic orbifolds (as is typical in the VOA and conformal field theory  literature), because central extensions $\Gamma$ of cyclic groups are always abelian. To get nonabelian $\Gamma$ here, we need to allow noncyclic $G$.

\section{Fixed-point free actions and $G$-Tambara-Yamagami}

The other extreme is when the module-map $\rho$ is fixed-point free, i.e.\ 
when $G$ acts without fixed points on $A$. We will find shortly that this corresponds to a braided $G$-crossed extension of \textbf{Vec}$_A^q$  generalizing the Tambara-Yamagami ones:

\begin{definition}\label{Gty} Let $A$ be a finite abelian group. By a \emph{$G$-Tambara-Yamagami category $\mathbf{C}$ of type} $\mathcal{TY}_G(A)$ we mean a fusion category whose only simple objects (up to equivalence) are $a\in A$ as well as  a unique $M^g$ for each $g\in G$, $g\ne e$, and which obeys the fusions 
\begin{equation}a\otimes a'\cong aa'\,,\ a\otimes M^g\cong M^g\otimes a\cong M^g,\ M^g\otimes M^h=\left\{\begin{matrix}\oplus_{a\in A}a&\mathrm{if}\ g=h^{-1}\cr \sqrt{|A|}M^{gh}&\mathrm{otherwise}\end{matrix}\right.\label{gTY}\end{equation}\end{definition}

For $G=\Z/2$ this recovers the standard Tambara-Yamagami categories \cite{TY}  (and nothing more). $\Z/3$-Tambara-Yamagami categories are discussed in \cite{JL}, where they used  \cite{ENO} to show they exist iff $A\cong A'\oplus A'$ for some abelian group $A'$.
Note that $\mathcal{TY}_G(A)$-type categories are automatically $G$-graded, with trivial component tensor-equivalent to \textbf{Vec}$_A^\omega$ for some $\omega\in Z^3(A,\bbC^\times)$, and with a unique $g$-graded simple for each $g\ne e$.  We are interested in braided $G$-crossed extensions of some MTC \textbf{Vec}$_A^q$, which are of $G$-Tambara-Yamagami type. \cite{JL} did not address the question of braided $\Z/3$-crossed structure on their categories.

For $A=0$ the $G$-Tambara-Yamagami categories are  $\mathbf{Vec}_G^\omega$ for some $\omega\in Z^3(A,\bbC^\times)$.

 The following is immediate:
\begin{lemma} \label{ty1} Suppose there is a category \textbf{C} of  type $\mathcal{TY}_G(A)$, which is a braided $G$-crossed extension of  some MTC $\mathbf{Vec}_A^q$. Then: 

\smallskip\noindent$\mathbf{(a)}$ $|A|\equiv 1$ (mod $|G|$) and, if $|G|>2$, $\sqrt{|A|}\in\bbZ$;

\smallskip\noindent$\mathbf{(b)}$ as a fusion category, $\mathbf{Vec}_A^q$ is $\mathbf{Vec}_A^\omega$ where $[\omega]$ is trivial in $H^3(A,\bbC^\times)$;

\smallskip\noindent$\mathbf{(c)}$ as long as $|A|>1$, $\mathbf{C}\not\cong\mathbf{Rep}(H)$ as a fusion category, for any finite group $H$;

\smallskip\noindent$\mathbf{(d)}$ there are exactly $|H^3(G,\bbC^\times)|$  braided $G$-crossed extensions of  $\mathbf{Vec}_A^q$ with the same action of $G$ on $A$ as $\mathbf{C}$ has.\end{lemma}

\begin{proof}  That $\sqrt{|A|}\in\bbZ$, follows from \eqref{gTY} for any choice $g,h\ne e$ with $gh\ne e$.
 Also, Proposition \ref{rankDg} says that each nontrivial $g\in G$ must act fixed-point freely on $A$ (since there is a unique $g$-graded simple). Thus each nontrivial $G$-orbit in $A$ has full cardinality $|G|$, i.e.\ $|A|\equiv 1$ (mod $|G|$). Part (b) is Lemma 2.6 in \cite{JL}. For part (d), the torsor over $H^2$ is trivial since $|A|$ and $|G|$ are coprime (by (a)), so by \cite{ENO} we only have the $H^3$ torsor.
 
 To prove (c), suppose for contradiction that $\mathbf{C}\cong\mathbf{Rep}(H)$. Note that $|H|=\sum_{a}1^2+\sum_{g\ne e}\sqrt{|A|}^2=|A|\,|G|$. Since \textbf{Rep}$(H)$ must be graded by the centre $Z(H)$ (more precisely, by $\widehat{Z(H)}\cong Z(H)$), we must have that $G\hookrightarrow Z(H)$ (since \textbf{C} is graded by $G$). Now, $H/[H,H]\cong A$ (the group of 1-dimensional representations of $H$), so $|[H,H]|=|G|$. Thus any $h\in [H,H]$ has order dividing $|G|$, so $G$ central implies $\langle h,G\rangle$ has order dividing both $|G|^\infty$ and $|H|=|A|\,|G|$, and we get $|\langle h,G\rangle|=|G|$ as $|A|$ is coprime to $|G|$. Hence $[H,H]=G$. Therefore $H$ is a central extension of $A$ by $G$, i.e.
 \begin{equation}e\to G\to H\to A\to 0\label{gpext2}\end{equation}
 (the flip of \eqref{gpext}!) But $H^2(G,A)=0$ since $|A|$ and $|G|$ are coprime, so Eilenberg-Mac Lane says \eqref{gpext2} splits. Then $G$ central forces $H\cong G\times A$, i.e.\ $H$ is abelian, contradicting the existence of $H$-irreps of dimension $\sqrt{|A|}$.\end{proof} 
 
 Because the obstruction $o_3$ and the $H^2$-torsor both vanish, the group extension $\Gamma$ promised by \eqref{gpext} is $A\sdprod G$. Such semidirect products, where $G$ acts fixed-point freely on $A$, are examples of Frobenius groups. We will return to this in Section 5.2.
  
\begin{theorem} \label{GTy} Suppose $\cV$ is a pointed VOA with $\mathbf{Mod}\,\,\cV=\mathbf{Vec}_A^q$, and $G$ a group of $\cV$-automorphisms. Suppose the  module-map $\rho$ acts fixed-point freely on $A$. Then the braided $G$-crossed extension $\mathbf{TwMod}_G\,\cV$ of $\mathbf{Vec}_A^q$  is $G$-Tambara-Yamagami. \end{theorem}

\begin{proof} We know \textbf{TwMod}$_G\,\cV$ is a  braided $G$-crossed extension  of $\mathbf{Vec}_A^q$, by Theorem \ref{voathm}. The uniqueness of the $g$-twisted module $M^g$ follows from Proposition \ref{rankDg}. The $G$-grading then requires $a\otimes M^g\cong M^g\otimes a\cong M^g$. By rigidity, $0\in M^g\otimes M^{g^{-1}}$ with multiplicity 1, so $a\in a\otimes M^g\otimes M^{g^{-1}}\cong M^g\otimes M^{g^{-1}}$, also with multiplicity 1. Hence $ M^g\otimes M^{g^{-1}}\cong \oplus_{a\in A}a$. Since each $a$ is invertible, each FPdim$(a)=1$, so each FPdim$(M^g)=\sqrt{|A|}$ and hence $M^g\otimes M^h\cong \sqrt{|A|} M^{gh}$ when $gh\ne e$. \end{proof}

We see from this argument why the multiplicity $\sqrt{|A|}$ appears in \eqref{gTY}. Of course \textbf{Mod}$\,\,\cV^G$ will be the $G$-equivariantization $( \mathbf{TwMod}_G\,\cV)^G$.

\subsection{Warm-up: standard Tambara-Yamagami}

The special case where $G=\Z/2$ is well-understood. The complete list of  fusion categories realizing \eqref{gTY} for $G=\Z/2$ are the Tambara-Yamagami categories $\mathbf{TY}_\pm(A,q)$ for some metric group $(A,q)$ \cite{TY}. Here, the sign $\pm$ should be regarded as an element of $H^3(\Z/2,\C^\times)\cong\Z/2$ (the torsor over $H^2$ is trivial).  

For $G=\Z/2$ and $|A|$ odd (the relevant choice for us when $G=\Z/2$), $\rho(a)=-a$ will be a fixed-point free module-map for any \textbf{Vec}$_A^q$. Indeed, it is the only such map  for  
$G=\Z/2$ (see e.g.\ Corollary 3.2 of \cite{Scho} for a generalization). So for $G=\Z/2$ (and $|A|$ odd), Theorem \ref{GTy} recovers nothing more than the standard Tambara-Yamagami.

For $|A|$ odd, the fusion category  $\mathbf{TY}_\pm(A,q)$ has a unique structure as a braided $\Z/2$-crossed category.
The modular data for the equivariantization $\mathbf{TY}_\pm(A,q)^{\Z/2}$ was computed in \cite{EGty}. There it was shown that for any sign and metric group $(A,q)$, a lattice VOA $\cV_L$ can be found with \textbf{Mod}$\,\,\cV_L=\mathrm{\textbf{Vec}}_A^{q}$,  and with a $\Z/2$-action on  $\cV_L$ for which \textbf{TwMod}$_{\Z/2}\,\cV_L= \mathbf{TY}_\pm(A,q)$.

A quasi-Hopf algebra realization of $\mathbf{TY}_\pm(A,q)$ requires $\sqrt{|A|}=\mathrm{FPdim}(M^1)$ to be an integer. \cite{Tam} shows that there is in fact a Hopf algebra realization iff $\sqrt{|A|}\in\Z$, the plus sign is chosen, and $q$ is `hyperbolic'.

\cite{GLM} studies $\Z/2$-orbifolds explicitly and in general, including some $A$ infinite examples, and how VOA extensions commute with the orbifold.  

\subsection{General results on fixed-point free actions on metric groups}

We begin with some applications of standard results on Frobenius groups. An accessible treatment of the latter is in \cite{Mayr}.

\begin{lemma} \label{ty2} Let $(A,q)$ be a metric group, and $\rho\,{:}\,G\to \mathrm{O}(A,q)$ be a fixed-point free homomorphism. 

\smallskip\noindent$\mathbf{(a)}$ Any subgroup of $G$ of square-free order, or of order $p^2$ ($p$ prime),   must be cyclic.

\smallskip\noindent$\mathbf{(b)}$ The only nilpotent $G$ are cyclic, or the direct product of an odd cyclic with generalized quaternions $Q_{2^k}=\langle  a,b|b^4=1,a^{2^{k-2}}=b^2,ab=ba^{-1}\rangle$.

\smallskip\noindent$\mathbf{(c)}$  We can write $\mathbf{Vec}_A^q$ as the Deligne product of $\mathbf{Vec}_{A_i}^{q_i}$ where $A_i=(\Z/p_i^{k_i})^{n_i}$ and $\rho$ restricts to fixed-point free homomorphisms $G\to\mathrm{O}(A_i,q_i)$.\end{lemma}


\begin{proof}  Part (a) is Corollary 3.18 resp.\  Theorem 3.7(a) in \cite{Mayr}. A nilpotent group is the direct product of its $p$-Sylow subgroups. So
part (b) follows from Theorem 3.7(b) in \cite{Mayr}. 

Write $A=\oplus_jA_{(j)}$, where $A_{(j)}$ is the (unique) $p_j$-Sylow subgroup of $A$ ($p_j$ being the distinct prime divisors of $|A|$). Let $q_j$ resp.\ $\rho_j$ be the restriction of $q$ resp.\ $\rho$ to $A_{(j)}$. Then each $(A_{(j)},q_j)$ is a metric group (its non-degeneracy etc is inherited from that of $q$). Therefore \textbf{Vec}$_A^q$ is the Deligne product of \textbf{Vec}$_{A_{(j)}}^{q_j}$. Moreover, each $\rho_j(g)$ sends $A_{(j)}$ to itself (since it must preserve order of elements), so $\rho_j\,{:}\,G\to\mathrm{O}(A_{(j)},q_j)$. Finally, each $\rho_j$ is fixed-point free since $\rho$ is. 

Thus to prove part (c) it suffices to consider when $A$ is a $p$-group. Let $p^k$ be the exponent of $A$. Lemma 6.4 of \cite{Mayr} says we can write $A= A_1\oplus B$ in such a way that $A_1\cong(\Z/p^{k})^{n}$ is homocyclic (i.e.\ a product of isomorphic cyclic groups) of maximal exponent, $\rho(g)(A_1)=A_1$ for all $g$, and $B$ has exponent $<p^k$. Let $\rho_1$ resp.\ $q_1$ be the restriction of $\rho$ resp.\ $q$ to $A_1$. Note that $q_1$ is non-degenerate (since $q$ is), so $(A_1,q_1)$ is a metric group and $\rho_1\,{:}\,G\to\mathrm{O}(A_1,q_1)$ is fixed-point free. Let $A_1^\perp$ be the complement in $A$ of $A_1$ with respect to $q$. Then $A_1\cap A_1^\perp=0$ so $A=A_1\oplus A_1^\perp$ and the exponent of $A_1^\perp$ is $<p^{k}$. Let $\rho'$ resp.\ $q'$ be the restriction to $A_1^\perp$ of $\rho$ resp.\ $q$. Then $(A_1^\perp,q')$ is a metric group, and $\rho'\,{:}\,G\to\mathrm{O}(A_1^\perp,q')$ is fixed-point free. Now repeat the argument for $A_1^\perp$ to define subgroups $A_2,A_3$ and so on. 
\end{proof}

Conversely, if  $\rho_i\,{:}\,G\to\mathrm{O}((\Z/p_i^{k_i})^{n_i},q_i)$ are fixed-point free homomorphisms on homocyclic groups, then their product gives a fixed-point free homomorphism $\rho\,{:}\,G\to\mathrm{O}(\oplus_i(\Z/p_i^{k_i})^{n_i},\prod_iq_i)$.
Thus it suffices to restrict attention henceforth to $A=(\Z/p^k)^n$. 

Incidentally, the generalized quaternions have cohomology groups $H^3(Q_{2^k},\bbC^\times)\cong\Z/2^k$ and $H^4(Q_{2^k},\bbC^\times)=0$.

Write $\xi_n=\exp(2\pi \mathrm{i}/n)$. The possible metric groups are classified in e.g.\ \cite{Niku}. For $p$ odd, fix a nonquadratic residue $m$ of $p$. Then up to MTC-equivalence any \textbf{Vec}$_{\Z/p^k}^q$ is equivalent  to either \textbf{Vec}$_{\Z/p^k}^+$ (with $q(\ell)=\xi_{p^k}^{\ell^2}$) or \textbf{Vec}$_{\Z/p^k}^-$ (with $q(\ell)=\xi_{p^k}^{m\ell^2}$). We can write any \textbf{Vec}$_{(\Z/p^k)^n}^q$ (up to equivalence) as a Deligne product of \textbf{Vec}$_{p^k}^{s_i}$, $s_i=\pm$. Since \textbf{Vec}$_{\Z/p^k}^{-}\boxtimes\mathbf{Vec}^-_{\Z/p^k}\cong \mathbf{Vec}_{\Z/p^k}^{+}\boxtimes\mathbf{Vec}^+_{\Z/p^k}$, we can insist that at most one of the signs be $-$.

When $A$ is a 2-group, the only \textbf{Vec}$_A^q$ compatible with Lemma \ref{ty1}(b) are Deligne products of copies of  \textbf{Vec}$_{(\Z/2^k)^2}^\pm$, where `$+$' corresponds to quadratic form $q_+(a,b)=ab/2^k$ and `$-$' corresponds to $q_-(a,b)=(a^2+ab+b^2)/2^k$.

\subsection{Realization of $G$-Tambara-Yamagami by VOAs}

As explainded last subsection, we can restrict to $A$ of the form $(\Z/p^k)^n$.
There are many $G$ which have braided $G$-crossed categories of type $\mathcal{TY}_G(A)$ -- e.g.\ the order 120 group SL$_2(\Z/5)$ is the smallest nonsolvable $G$ for which a fixed-point free $\rho\,{:}\,G\to\mathrm{O}(A,q)$ can be found.  But thanks to Lemma \ref{ty2} most of the ones of smaller order will be cyclic or a direct product of an odd cyclic with generalized quaternions. In this subsection we restrict to these $G$, and find for each prime power $p^k$ an $n\ge 1$ and a lattice $L$ such that \textbf{Mod}$\,\,\cV_L\cong\mathbf{Vec}_{(\Z/p^k)^n}^q$ for some $q$, and a group of $\cV_L$-automorphisms isomorphic to $G$ whose module-map $\rho\,{:}\,G\to \mathrm{O}((\Z/p^k)^n,q)$ is fixed-point free. Then by Theorem \ref{GTy}, \textbf{TwMod}$_G\,\cV_L$ will be a braided $G$-crossed extension of \textbf{Vec}$_{(\Z/p^k)^n}^q$ of type $\mathcal{TY}_G(A)$. Our suspicion is that the $n$ we find may be the  smallest possible. 

The easiest source of such VOA-actions are \textit{signed permutation representations}. This is a matrix representation $\rho$ where each $\rho(g)$ is a signed permutation matrix, i.e.\ \begin{equation}\label{spr}\rho(g)(u_1,u_2,\ldots,u_n)=(s_{g,1}u_{\pi_g^{-1}1},s_{g,2}u_{\pi_g^{-1}2},\ldots,s_{g,n}u_{\pi_g^{-1}n})\end{equation}  where $\pi_g\in\mathrm{Sym}(n)$ and each $s_{g,i}\in\{\pm 1\}$.  We say $\rho$ is fixed-point free if no $\rho(g)$, for $g\ne e$, has 1 as an eigenvalue.

\begin{lemma} \label{L:permrep}  Let $\rho$ be an $n$-dimensional signed permutation representation of a finite group $G$. Choose any MTC $\mathbf{Vec}_A^q$, where $|A|$ is coprime to $|G|$. Then there is a lattice VOA $\cV$ with $\mathbf{Mod}\,\,\cV\cong(\mathbf{Vec}_A^q)^{\boxtimes n}:=\mathbf{Vec}_A^q\boxtimes\cdots\boxtimes\mathbf{Vec}_A^q$ ($n$ times) as a MTC, and $G$ acts as automorphisms on $\cV$ with  module-map given by $\rho$. If $\rho$ as a signed permutation representation is fixed-point free, then so is the module-map.\end{lemma}

\begin{proof} Let $L$ be any even positive definite lattice realizing  $(A,q)$. Choose the 2-cocycle $\varepsilon_L$ defining $\cV_L$ using \eqref{eps}. Note that $\varepsilon_L(-u,-v)=\varepsilon_L(u,v)$ for this choice. Define $L^{\oplus n}:=L\oplus\cdots\oplus L$ ($n$ times) for any $n$. We can take the 2-cocycle for $\cV_{L^{\oplus n}}$ to be $\varepsilon_{L^{\oplus n}}((u_1,...,u_n),(v_1,...,v_n))=\prod_i\varepsilon_L(u_i,v_i)$. Then \textbf{Mod}$\,\,\cV_{L^{\oplus n}}\cong(\mathbf{Vec}_A^q)^{\boxtimes n}$ as MTC. 

Let $g\in G$. By assumption $\rho(g)$ is a signed permutation matrix. Use it to define an isometry $\sigma=\sigma_g$ of $\Lambda^{\oplus n}$ through \eqref{spr}.  To study the possible lifts $\hat{\sigma}$ to Aut$(\cV_{L^{\oplus n}})$, we need to solve \eqref{eta} for $\eta$. But  \begin{eqnarray}\frac{\varepsilon_{L^{\oplus n}}(\sigma(u_1,...,u_n),\sigma(v_1,...,v_n))}{\varepsilon_{L^{\oplus n}}((u_1,...,u_n),(v_1,...,v_n))}\!\!&=\prod_{i=1}^n\frac{\varepsilon_L(s_1u_{\pi^{-1}1},s_1v_{\pi^{-1}1})\varepsilon_L(s_2u_{\pi^{-1}2},s_2v_{\pi^{-1}2})\cdots\varepsilon_L(s_nu_{\pi^{-1}n},s_nv_{\pi^{-1}n)}}{\prod_{j=1}^n\varepsilon_L(u_j,v_j)}\nonumber\\
&=\prod_{i=1}^n\frac{\varepsilon_L(s_{\pi i}u_i,s_{\pi i}v_i)}{\varepsilon_L(u_i,v_i)}=1\nonumber\end{eqnarray}
This means we can choose $\eta$ in \eqref{eta} to be identically 1, for all $g\in G$. Then manifestly $\widehat{\rho(g)}$ has the same order as $\rho(g)$, in fact $\widehat{\rho}$ defines a homomorphism of $G$ into Aut$(\cV_{L^{\oplus n}})$, so $\widehat{\rho}$ defines an action of $G$ on $\cV=\cV_{L^{\oplus n}}$ by VOA automorphisms. The module-map of $\widehat{\rho}$ will also be by \eqref{spr}:
 \begin{equation}\label{spr2}\widehat{\rho}(g)(a_1,a_2,\ldots,a_n)=(s_{g,1}a_{\pi_g^{-1}1},s_{g,2}a_{\pi_g^{-1}2},\ldots,s_{g,n}a_{\pi_g^{-1}n})\ \  \forall (a_1,a_2,...,a_n)\in(\mathbf{Vec}_A^q)^{\boxtimes n}\end{equation} 
 for the same $\pi_g$ and $s_{g,i}$ as for $\rho(g)$.   \end{proof}

The only finite groups $G$ with a fixed-point free signed permutation representation are 2-groups; 
by Lemma \ref{ty2}, such a $G$ must then be either cyclic or generalized quaternion. 
Consider first the cyclic group $G=\bbZ/2^k$. It has a fixed-point free $2^{k-1}$-dimensional signed permutation representation generated by the $2^{k-1}$-cycle $\pi_1=(12\cdots 2^{k-1})$ and signs $s_{1,i}=1$ for $i>1$ and $s_{1,1}=-1$.

Next consider the generalized quaternions  $Q_{2^k}$ of order $2^k$ ($k>2$), with presentation given in Lemma \ref{ty2}(b). 
To the generator $a$ resp.\ $b$ associate the signed permutation matrices with
\begin{eqnarray}\pi_a=(12\cdots 2^{k-2})(2^{k-2}+1,2^{k-2}+2,\ldots, 2^{k-1})\,,\ \ s_{a,1}=s_{a,2^{k-2}+1}=-1\,,\\
\pi_b=(1,2^{k-1}+1)(2,2^{k-2}+2)\cdots(2^{k-2},2^{k-1})\,,\ \ s_{b,j}=-1\ \forall j\le 2^{k-2}\,,\end{eqnarray}
using cycle notation, and all other $s_{a,i},s_{b,j}=1$. Again, this is fixed-point free.

\begin{corollary} \label{TYeven} Choose any metric group $(A,q)$ ($|A|$ odd) and let $G=\Z/2^k$ or $Q_{2^k}$ for some $n$. Then there is a braided $G$-crossed extension $\mathbf{C}$ of $(\mathbf{Vec}_A^q)^{\boxtimes 2^{k-1}}$ of $G$-Tambara-Yamagami type, and a lattice VOA $\cV$ on which $G$ acts as automorphisms, for which $\mathbf{TwMod}_G\,\cV\cong\mathbf{C}$.\end{corollary}

The existence of \textbf{C} had already been clear from the vanishing of the obstructions $[o_3],[o_4]$. Corollary \ref{TYeven} gives the explicit module-map (given by the appropriate signed permutation matrix). The deeper part of the Corollary is realizing \textbf{C} by a VOA, which was done in Lemma \ref{L:permrep}.

To handle odd $|G|$, we have to generalize beyond signed permutation representations. Any group $G$ of isometries of an $n$-dimensional lattice $L$ defines an $n$-dimensional integral representation of $G$. The irreducible \textit{integral} representations of $G=\Z/n$  (as opposed to the irreducible representations over say $\bbC$) are in natural bijection with the divisors $d|n$, where the generator $1\in G$ is sent to the companion matrix of the cyclotomic polynomial $\Phi_d(x)$, and thus has dimension $\phi(d)$ in terms of Euler's $\phi$-function (Theorem 73.9 of \cite{CR}). The only of these which is faithful has $d=n$. There are other indecomposable integral representations: e.g.\ for $G=\Z/p$ there are precisely $2h_p+1$ indecomposables (and only 2 irreducibles) where $h_p$ is the class number of the cyclotomic field $\Q[\xi_p]$ (see Theorem 74.3 of \cite{CR}). But the smallest integral representation of $G=\Z/n$ with no fixed points is that irreducible of dimension $\phi(n)$. We will construct our $\Z/n$-actions on VOAs using that irrep. 

Choose any odd prime $p$. The root lattice $A_{p-1}$ can be identified with the subset of all $(u_1,u_2,...,u_p)\in\Z^p$ with $\sum_iu_i=0$. Note that $A_{p-1}$ has an order $p$ isometry $\sigma$ defined by $\sigma(u_1,...,u_p)=(u_p,u_1,...,u_{p-1})$. Each power $\sigma^\ell$, $0<\ell<p$, is fixed-point free (apart from $\vec{0}=(0,..,0)$ of course). For any $k\ge 1$ let $\Lambda_{p^k}$ be the lattice $A_{p-1}^{\oplus p^{k-1}}$ (i.e.\ orthogonal direct sum of $p^{k-1}$ copies of $A_{p-1}$) and let $\sigma_{p^k}$ be the map on $\Lambda_{p^k}$ given by
\begin{equation}\label{E:fpf}(\vec{v}_1,\vec{v}_2,\ldots,\vec{v}_{p^{k-1}})\mapsto(\sigma(\vec{v}_{p^{k-1}}),\vec{v}_1,\ldots,  \vec{v}_{p^{k-1}-1})\end{equation}
Then $\sigma_{p^k}$ has order exactly $p^k$ and is an isometry of $\Lambda_{p^k}$. Moreover, no power $\sigma_{p^k}^\ell$, $0<\ell<p^k$, has fixed points: to see this,  it suffices to compute $\sigma_{p^k}^{p^{k-1}}(\vec{v}_1,...,\vec{v}_{p^{k-1}})=( \sigma(\vec{v}_1),...,\sigma(\vec{v}_{p^{k-1}}))$.

Recall that $A_{m-1}^*/A_{m-1}\cong\Z/m$ for any $m>1$, generated by $\Lambda^{(m)}_1=(\frac{m-1}{m},\frac{-1}{m},...,\frac{-1}{m})$ using the $A_{m-1}\subset \Z^m$ realization, with norm $\Lambda_1^{(m)}\cdot\Lambda^{(m)}_1=\frac{m-1}{m}$. Note that $\sigma(\Lambda^{(m)}_1)\in \Lambda^{(m)}_1+A_{m-1}$. The lattice isometry group $O(A_{m-1})\cong \mathrm{Sym}(m)\times\{\pm 1\}$.

\begin{lemma} \label{L:fpf} Choose any $G=\Z/p^k$ ($p$ odd) and any $n\in\Z_{>0}$ coprime to $p$.  Define a metric group $(A,q)$ on $A=\frac{1}{\sqrt{n}}\Lambda_{p^k}/{\sqrt{n}}\Lambda_{p^k}\cong (\Z/n)^{\oplus \phi(p^k)}$ using the quadratic form of $\Lambda_{p^k}$. Then there is a braided $G$-crossed category $\mathbf{C}$ of type $\mathcal{TY}_G(A,q)$, and a lattice VOA $\cV$ on which $G$ acts, such that $\mathbf{TwMod}_G\,\cV\cong\mathbf{C}$. The action \eqref{E:fpf} of $G$ on $\Lambda_{p^k}$ descends to the same action of $G$ on $A$.\end{lemma}
 
 \begin{proof} Fix $p$ and $n$ as above. Consider first the case where $k=1$. Begin by noting that $(\sqrt{n}A_{p-1})^*/(\sqrt{n}A_{p-1})\cong (\Z/n)^{p-1}\oplus\Z/p$ as groups, since $n$ and $p$ are coprime.  The $\Z/p$ part is spanned by $\sqrt{n}\Lambda_1^{(p)}$ and corresponds to metric group $(\Z/p,\left(\frac{-2n}{p}\right))$ using the discussion at the end of Section 5.2.  The $ (\Z/n)^{p-1}$ part comes from $A:=(\frac{1}{\sqrt{n}}A_{p-1})/(\sqrt{n}A_{p-1})$ and is the orthogonal complement to $\Z/p$ (so is necessarily non-degenerate). Write $q$ for its quadratic form. After this proof we  discuss what $q$ is, but this isn't important for now.
 
 Choose any prime $p'$ satisfying $\left(\frac{p'}{p}\right)=\left(\frac{2n}{p}\right)=\left(\frac{2}{p'}\right)$ and $p'\equiv -p$ (mod 4), using Legendre symbols. These conditions force $p'$ to lie in one of $(p-1)/2$ different congruence classes mod $8p$, so by Dirichlet's Theorem there are infinitely many $p'$ to choose from. These conditions on $p'$ mean
that there are integers $a,b$ such that \begin{equation}n\equiv 2p'a^2\ (\mathrm{mod}\ p)\,,\ \ \  1\equiv 2pb^2\ (\mathrm{mod}\ p')\label{eq:ab}\end{equation} where we use quadratic reciprocity and multiplicativity of the Legendre symbols.
 
 Let $\Lambda^{(p)}$ denote the $\Z$-span of the lattice $A_{p-1}\oplus \sqrt{2pp'}\Z\oplus A_{p'-1}\oplus\sqrt{2}\Z$ with the dual lattice vectors $(\sqrt{n}\Lambda_1^{(p)},\frac{2p'a}{\sqrt{2pp'}},0,0),(0,\frac{2pb}{\sqrt{2pp'}},\Lambda^{(p')}_1,0),(0,\frac{pp'}{\sqrt{2pp'}},0,\frac{1}{\sqrt{2}})$ (this gluing theory of lattices is discussed in Section 4.3 of \cite{CS}, and corresponds in VOA language to simple current extensions of lattice VOAs). Then, using \eqref{eq:ab},  $\Lambda$ is an even positive definite lattice, with metric group naturally isomorphic to $ (A,q)$. $\sigma$ is an isometry of $\Lambda$, where we have it fix the 
 $ \sqrt{2pp'}\Z\oplus A_{p'-1}\oplus\sqrt{2}\Z$ part. This action of $\Z/p$ descends to the corresponding fixed-point free action on $A$.
 
 Let $\cV=\cV_{\Lambda^{(p)}}$. By Lemma \ref{voaaut}, the isometry $\sigma$ lifts to an order $p$ automorphism of $\cV$. The rest follows from Theorem \ref{GTy}.
 
 The argument for $p^k$ ($k>1$) is similar. $\Lambda^{(p^k)}$  will be $p^{k-1}$ copies of $\Lambda^{(p)}$. The isometry $\sigma_{p^k}$ will permute these copies, as well as twist the first $A_{p-1}$ by $\sigma$ as in \eqref{E:fpf}. It lifts to an order $p^k$ automorphism of the VOA $\cV=\cV_{\Lambda^{(p^k)}}$, whose module-map looks like \eqref{E:fpf} and is fixed-point free. Theorem \ref{GTy} concludes the proof.\end{proof}

If $n$ is coprime to $p!$, then the quadratic form $q$ on $A\cong(\Z/n)^{\phi(p^k)}$ can be obtained by noting $\alpha_1,\alpha_1+2\alpha_2,...,\alpha_1+2\alpha_2+\cdots+(p-1)\alpha_{p-1}$ are an orthogonal set of generators for $A$, where the $\alpha_i$ are the simple roots of $A_{p-1}$, and these have norm $2,6,...,(p-1)p$. For example, if $p=3$ and $n=p^{\prime k}$ where $p'>3$ is  prime then \textbf{Mod}$\,\,\cV\cong \mathbf{Vec}_{(\Z/p^{\prime k})^2}^{+,\pm}$ using notation from the end of Section 5.2, where we take $+,+$ if $p'\equiv\pm 1$ (mod 12) and $+,-$ otherwise. If $n$ is not coprime to $p!$, then those won't be generators and we must be more careful. For example, when $n=2$ and $p$ is arbitrary, \textbf{Mod}$\,\,\cV$ is the Deligne product of $\phi(p^k)/2$ copies of $\mathbf{Vec}_{(\Z/2)^2}^+$: the $\ell$th copy of $((\Z/2)^2,q_+)$ in $\frac{1}{\sqrt{2}}A_{p-1}/\sqrt{2}A_{p-1}$ has generators $\sum_{i=0}^{\ell-1}\alpha_i$ and $\alpha_{2\ell}+\sum_{i=0}^{\ell-1}\alpha_i$.

Because Lemma \ref{L:fpf} is based on the smallest faithful integral representations, we suspect the number $n=\phi(p^k)$ there cannot be improved.
Nevertheless, our list of VOA realizations of possible braided $G$-crossed $G$-Tambara-Tamagami categories for that $n$ is not complete. For example, for $n=2$ and $G=\Z/3$ the theorem obtains type $\mathcal{TY}_{\Z/3}((\Z/2)^2,+)$ (recall the metric groups for 2-groups discussed at the end of Section 5.2). But through triality $\Z/3$ also acts without fixed points on $D^*_4/D_4$, yielding type $\mathcal{TY}_{\Z/3}((\Z/2)^2,-)$.


It is now elementary to find VOA realizations for any nilpotent $G$ allowed by Lemma \ref{ty2}(b), by taking the tensor product of the corresponding lattices and isometries. For example, for $G=\Z/15$ take $\cV=\cV_{\Lambda^{(3)}\otimes\Lambda^{(5)}}$. Identify $\Lambda^{(3)}\otimes\Lambda^{(5)}$ with the vectors $(v_0,...,v_{14})\in\Z^{15}$ satisfying $\sum_{i=0}^2v_{k+5i}=0=\sum_{j=0}^4v_{m+3j}$ for all $ 0\le k\le4$ and all $0\le m\le 2$. Then the order 15 isometry sends $v_i\mapsto v_{i+1\ (\mathrm{mod}\ 15)}$. For
$G=\Z/6$, take $\cV=\cV_{\Lambda^{(3)}}$ and order 6 isometry taking $(a_1,a_2,a_3)\in \frac{1}{\sqrt{n}}A_2/\sqrt{n}A_2$ to $(-a_3,-a_1,-a_2)$. For $G=\Z/p^k\times Q_8$ the VOA will be $\cV_{(\Lambda^{(p^k)})^{\oplus 4}}$, etc.

\subsection{Simples and modular data of the $G$-equivariantization \textbf{Mod}$\,\,\cV^G$}

Once we have the braided $G$-crossed category \textbf{TwMod}$_G\,\cV$, it is easy to get information on the corresponding orbifold category \textbf{Mod}$\,\,\cV^G$, which in tensor category language is the $G$-equivariantization of \textbf{TwMod}$_G\,\cV$. When there is a lattice VOA relization, and $G$ is cyclic, this been studied in the interesting paper \cite{BEKT}. In their notation, $G$-Tambara-Yamagami means $\mathfrak{h}_0=\bbC$ so there is no $h$-dependence on their characters, and they will usually be linearly dependent, limiting the validity of the proposed modular data. 

Let $\mathbf{C}$ be a braided $G$-crossed extension of $\mathbf{Vec}_A^q$ of type $\mathcal{TY}_G(A)$. In the case of a pointed VOA realization, Mod$\,\,\cV\cong\mathbf{Vec}_A^q$, $\mathbf{TwMod}_G\,\cV\cong\mathbf{C}$, and $\mathbf{Mod}\,\,\cV^G$ will be the equivariantization $\mathbf{C}^G$ of $\mathbf{C}$. Let $\rho:G\to \mathrm{O}(A,q)\le\mathrm{Aut}(A)$ be the module-map. Moving around the $H^3$-torsor, i.e.\ choosing a different $\omega\in H^3(G,\bbC^\times)$, multiplies  the associativity isomorphisms $(M^g\otimes M^{h})\otimes M^{k}\cong M^g\otimes(M^{h}\otimes M^{k})$ by $\omega$, and also changes the group action, e.g.\ 
by multiplying the isomorphism $(hk).M^g\to h.(k.M^g)$ by $\theta_g(h,k)$.

Much of the story is immediate from the definition of equivariantization. The simples of $\mathbf{C}^G$ correspond to the $G$-orbits, together with projective representations of the stabilizers. In particular, up to equivalence they are $[g,\chi]$ for each conjugacy class representative $g$ in $G$, where $\chi$ runs through the irreducible projective characters of the centralizer $C_G(g)$ with some 2-cocycle $\gamma\in Z^2(G,\mathbb{C}^{\times})$, together with one simple $[a]$ for each representative $a$ of a $G$-orbit $\rho(G).a$ in $A$, for $a\ne 0$. There will be precisely $(|A|-1)/|G|$ simples of type $[a]$ for $a\ne0$; if $G\cong\bbZ_n$ or $\bbZ_n\times\bbZ_{n'}$, the number of the remaining simples $[g,\chi]$ will equal the rank of $\cZ(\mathbf{Vec}_G)$.

Likewise, we can read off the restrictions $\mathbf{C}\to\mathbf{C}^G$. They are Res$\,a=[a]$ for $a\ne 0$, Res$\,0=\oplus_{\chi\in\mathrm{Irr}(G)}\mathrm{dim}\,\chi\,[e,\chi]$, and Res$\,M^g=\oplus_\chi\mathrm{dim}\,\chi\,[g,\chi]$ for $g\ne e$, where $\chi$ runs over the irreducible projective characters of $C_G(g)$ for the appropriate 2-cocycle. Therefore inductions are Ind$\,[a]=\oplus_{a'\in\rho(G).a}a'$, Ind$\,[e,\chi]=\mathrm{dim}\,\chi\,\, 0$ and Ind$\,[g,\chi]=\oplus_{h\in K_g}\mathrm{dim}\,\chi\, M^h$ for $g\ne e$, where $K_g$ is the conjugacy class of $g$ in $G$. Some of these formulas are familiar to those from $\cZ(\mathbf{Vec}_G^\omega)$ which we saw in Section 3, because formally these equivariantizations are similar. However induction, being a tensor functor, respects FPdim, and so we obtain FPdim$\,[g,\chi]=\sqrt{|A|}\,\|K_g\|\,\mathrm{dim}\,\chi$ for $g\ne e$, FPdim$\,[e,\chi]=\mathrm{dim}\,\chi$ and FPdim$\,[a]=|A|$. The extra factor of $\sqrt{|A|}$ is because FPdim$\,M^g=\sqrt{|A|}$.  Note that these FPdim's are all integers, at least when $|A|>2$.

Note that Res$\,0$ is an \'etale algebra, as always -- a copy $B_G$ of the regular representation of $G$.

The ribbon twists are $\theta([e,\chi])=1$, $\theta([g,\chi])=\chi(g)/\chi(e)$ for $g\ne e$, and $\theta([a])=q(a)$, the quadratic form on $A$. This is consistent with restriction preserving the ribbon twists of local objects (i.e.\ $a\in \mathbf{Vec}_A^q$).

We will stop here, but to compute the $S$ matrix, follow the method of \cite{GNN} which describes the general method and works it out for $G=\Z/2$.

For a simple example, consider the $\Z/3$-orbifold (by `triality') of $\cV_{D_4}$. We see that the 10 simples have FPdim's 1 (3 times), 2 (6 times) and 3 (once). \cite{BEKT} in Section 7.2 obtain the ribbon twists are 1,1,1, $\xi_9,\xi_9,\xi_9^4,\xi_9^4,\xi^7_9,\xi_9^7$, and $-1$, respectively, using lattice VOA methods. Then from our analysis we find that the ribbon twists for $g\ne e$ correspond to a nontrivial gauge anomaly $\omega\in Z^3(G,\bbC^\times)$.  \cite{BEKT} are unable to obtain the $S$ matrix from lattice VOA methods, but the $S$ matrix entries $S_{[g,\chi],[g',\chi']}$ for $g,g'\ne e$ (these are the difficult ones) for $\cV_{D_4}^{\Z/3}$ match those of $\cZ(\mathbf{Vec}_G^\omega)$ for that $\omega$, up to a factor independent of $g,g',\chi,\chi'$, and so can be written down explicitly.

\section{The diagonal argument}

In this section we explore a simple strategy to combine the trivial module-map (i.e.\ cleft) case and the fixed-point free module-map ($G$-Tambara-Yamagami) case into a hybrid model. Certainly this can be pushed much further than we do here.

The following is clear.

\begin{lemma} \label{L:diag} $\mathbf{(a)}$ Suppose $\mathbf{D}$ is a braided $G$-crossed extension of $\mathbf{C}$, and $\mathbf{D'}$ is a braided $H$-crossed extension of $\mathbf{C'}$.  Then $\mathbf{D}\boxtimes \mathbf{D'}$ is a braided $G\times H$-crossed extension of $\mathbf{C}\boxtimes\mathbf{C'}$. 

\smallskip\noindent$\mathbf{(b)}$ Let $N$ be a normal subgroup of a finite group $G$. Let $\mathbf{D}$ be a braided $G$-crossed extension of $\mathbf{C}$. Define the full subcategory $\mathbf{D}|_N$ of $\mathbf{D}$ whose objects are sums of homogeneous simple objects with grading in $N$. Then $\mathbf{D}|_N$ is a braided $N$-crossed extension of $\mathbf{C}$.
\end{lemma}

For example, write $\Delta_G$ for the diagonal subgroup $\{(g,g)\in G\times G\,|\, g\in G\}$. Then Lemma \ref{L:diag} says that $(\mathbf{D}\boxtimes\mathbf{D'})|_{\Delta_G}$ is a braided $G$-crossed extension of $\mathbf{C}\boxtimes \mathbf{C'}$. Perhaps a better name for this diagonal construction is a \textit{graded Deligne product} $\mathbf{D}\boxtimes_{(G)}\mathbf{D'}$. Note that it decomposes as $\mathbf{D}\boxtimes_{(G)}\mathbf{D'}=\oplus_{g\in G}\mathbf{D}_g\boxtimes\mathbf{D'}_g$.

 On VOAs the diagonal construction works as follows. Suppose $G$ acts as automorphisms on two VOAs $\cV$ and $\cW$. On $\mathbf{Mod}\,\,\cV$ suppose the module-map is trivial, so $\mathbf{TwMod}_G\,\cV$ is cleft, and on $\mathbf{Mod}\,\,\cW$ suppose it is fixed-point free, so $\mathbf{TwMod}_G\,\cW$ is $G$-Tambara-Yamagami. Then $\mathbf{TwMod}_{\Delta_G}\,(\cV\otimes\cW)=\mathbf{TwMod}_G\,\cV\boxtimes_{(G)}\mathbf{TwMod}_G\,\cW$ is hybrid.

More interesting is the converse:

\begin{theorem} \label{diagonal} Let $(A,q)$ be a metric group and $\mathbf{D}$ a braided $\Z/p$-crossed extension of $\mathbf{Vec}_A^q$ for some prime $p$ coprime to $|A|$. Let $B$ consist of all $a\in A$ fixed by all $g\in\Z/p$. Suppose that $q|_B$ is non-degenerate. Then $\mathbf{D}$  is equivalent, as a  braided $\Z/p$-crossed extension, to the graded Deligne product
\begin{equation}
\mathbf{D}_c\boxtimes_{(\Z/p)} \mathbf{D}_{ty}= (\mathbf{D}_c\boxtimes \mathbf{D}_{ty})|_{\Delta_{\Z/p}}
\end{equation}
where $\mathbf{D}_c$ is a pointed (i.e.\ cleft) braided  $\Z/p$-crossed extension of $\mathbf{Vec}_B^{q|_B}$ and  $\mathbf{D}_{ty}$ is a $\Z/p$-Tambara-Yamagami type  braided $\Z/p$-crossed extension of the MTC $\mathbf{Vec}_{B^\perp}^{q|_{B^\perp}}$. 
\end{theorem}
\begin{proof}  Note that $B$ is a subgroup of $A$, so non-degeneracy of $q|_B$ implies that $A=B\oplus B^\perp$. Also, $q|_{B^\perp}$ will be non-degenerate, so indeed $\mathbf{Vec}_{B^\perp}^{q|_{B^\perp}}$ is a MTC.  Moreover, by definition of $B$, the $\Z/p$-action on $B^\perp$ will be fixed-point free.  Let $\rho$ denote the module-map $\Z/p\to\mathrm{O}(A,q)$ corresponding to $\mathbf{D}$, and let $\rho_\perp$ denote the (fixed-point free) map $\Z/p\to\mathrm{O}(B^\perp,q|_{B^\perp})$.

Both obstructions $o_3,o_4$ vanish for the trivial map $\Z/p\to\mathrm{O}(B,q|_B)$ as well as for $\rho_\perp$, since the corresponding $H^3,H^4$ cohomology groups vanish, and therefore there exist braided $\Z/p$-crossed extensions of $\mathbf{Vec}_B^{q|_B}$ and $\mathbf{Vec}_{B^\perp}^{q|_{B^\perp}}$ for the respective module-maps. Choose one such extension for each. The former will be a pointed (cleft) category we'll call $\mathbf{D}_c$, and the latter, which we'll call $\mathbf{D}_{ty}$, will be of $\Z/p$-Tambara-Yamagami type.

Therefore by Lemma \ref{L:diag} there will be a braided $\Z/p$-crossed extension $\mathbf{D}_c\boxtimes_{(\Z/p)}\mathbf{D}_{ty}$ of $\mathbf{Vec}_A^q$, with module-map $\rho$. Now, the $H^2$-torsor for $\rho$ is trivial, since $H^2(\Z/p,A)=0$. Therefore all braided $\Z/p$-crossed extensions of $\mathbf{Vec}_A^q$ lie on a torsor over $H^3(\Z/p,\bbC^\times)\cong\Z/p$. Thus the original category $\mathbf{D}$ will be an $H^3$-twist of the associativity of  $\mathbf{D}_c\boxtimes_{(\Z/p)}\mathbf{D}_{ty}$.

On the other hand we see that there are precisely $p$  $\Z/p$-Tamabara-Yamagami  braided $\Z/p$-crossed extensions of $\mathbf{Vec}_{B^{\perp}}^{q|_{B^\perp}}$, related to each other through the $H^3$-torsor as well. By the linearity of the Deligne product, twisting $\mathbf{D}_{ty}$ by $\omega$ and then taking the diagonal product with $\mathbf{D}_c$ is the same as twisting $\mathbf{D}_c\boxtimes_{\Z/n}\mathbf{D}_{ty}$ by $\omega$. This means running through the $H^3$-torsor of $\rho$ gives only diagonal products of $\Z/p$-Tambara-Yamagami and a pointed extension, so making a different choice of $\mathbf{D}_{ty}$ if necessary, we get the theorem.
\end{proof}

For example, if $B$ is any subgroup of $A$ with $|A|/|B|$ coprime to $|B|$, then automatically $q|_{B}$ will be non-degenerate. 
 
Within  the VOA literature, \cite{BE} considers a $G=\Z/2$ example  of a similar argument. They consider an order-2 isometry $\sigma$ of a lattice $L$, and decompose $\sigma$ into its eigenspaces: let $L_\pm$ be the largest sublattices of $L$ such that $\sigma|_{L_\pm}$ acts like $\pm 1$. Then $L_+\oplus L_-$ is a sublattice of $L$ of full dimension, and $\cV_L^\sigma$ is a simple current extension of $\cV_{L_+\oplus L_-}^\sigma=\cV_{L_+}\otimes\cV_{L_-}^{\{\pm1\}}$.
 
\appendix

\section{The Naidu and Mason--Ng quasi-Hopf algebras}

\subsection{Definitions and constructions}
\begin{definition}{$G$-crossed Module}\label{Gcm}

A \emph{$G$-crossed module} is a triple $(G,K,\partial)$, where $G$ acts on the left of $K$, and a group homomorphism $\partial\,{:}\,K\rightarrow G$ such that for all $g\in G, h,h_1,h_2\in K$:
\begin{eqnarray}
\label{equivariant_conjugation}
\partial(\prescript{g}{}{h})=g\partial(h)g^{-1}
\\
\label{Peiffer_Identity}
\prescript{\partial(h_1)}{}{h_2}=h_1h_2h_1^{-1}
\end{eqnarray}
\end{definition}
We're ultimately interested in $G$-crossed modules coming from central extensions $K$ of $G$ by ker$\,\partial$.
 
Fix a $G$-crossed module $\mathfrak{X}=(G,K,\partial)$. If $\omega\in C^i(K,\mathbb{\C}^{\times})$, then $G$ acts on the right by  $\omega^g(x_1,\ldots,x_i)=\omega(\prescript{g}{}{x_1},\ldots ,\prescript{g}{}{x_i})$ for $g\in G$. We recall the following notion from \cite{Naidu}: 
\begin{definition}{\cite{Naidu}}
\label{qa3c_definition}
A quasi-abelian 3-cocycle is a tuple 
$(\omega,\gamma,\mu,c)$ such that:
\[\omega\in Z^3(K,\mathbb{C}^{\times}), \gamma\in C^2(G,C^1(K,\C^{\times})), \mu\in C^1(G,C^2(K,\mathbb{C}^{\times})), c\in C^2(K,\mathbb{C}^{\times})\]
and the following conditions hold: 
\begin{eqnarray}\gamma_{g,h}(^kx)\gamma_{gh,k}(x)=\gamma_{h,k}(x)\gamma_{g,hk}(x)\,, &x\in K, \ g,h,k\in G\\
d(\mu_g)=\frac{\omega^g}{\omega}\,, &g\in G\\
d(\gamma_{g,h})=(d\mu)_{g,h}\,, &g,h\in G\\
\frac{c^g(x,y)}{c(x,y)}=\frac{\mu_g(xyx^{-1},x)\gamma_{g\partial(x)g^{-1},g}(y)}{\mu_g(x,y)\gamma_{g,\partial(x)}(y)}\,, & g\in G, x,y\in K 
\\ 
c(xy,z)=\frac{\omega(x,y,z)\omega((xy)z(xy)^{-1},x,y)}{\omega(x,yzy^{-1},y)\gamma_{\partial(x),\partial(y)}(z)}c(x,yzy^{-1})c(y,z)\,,&x,y,z\in K \label{eqn:qa3c_braiding_1}\\
c(x,yz)=\frac{\omega(xyx^{-1},x,z)}{\omega(x,y,z)\omega(xyx^{-1},xzx^{-1},x)\mu_{\partial(x)}(y,z)}c(x,y)c(x,z) \,,&x,y,z\in K \label{eqn:qa3c_braiding_2}\end{eqnarray}
\end{definition}

\begin{notation} \emph{The set of quasi-abelian $3$-cocycles of the crossed module $\mathfrak{X}$ is denoted by $H^3_{qa}(\mathfrak{X},\mathbb{C}^{\times})$}. A quasi-abelian $3$-cocycle $(\omega,\gamma,\mu,c)$ is called normalized if $\omega,\gamma,\mu,c$ have the property that if any of their variables is the identity they will equal $1$. By defining a suitable notion of coboundary, it can be shown that there is no loss in generality by assuming that $(\omega,\gamma,\mu,c)$ is normalized, which we do from now on. See \cite{Naidu} for details.
\end{notation}
Associated to this quasi-abelian $3$-cocycle is a quasitriangular quasi-Hopf algebra: 

\begin{definition}\emph{Naidu quasi-Hopf Algebras}\label{naiqH}

Let $(G,K,\partial)$ be a $G$-crossed modules and $(\omega,\gamma,\mu,c)$ a quasi-abelian $3$-cocycle. Then $\mathbb{C}^{K}\otimes_{\mathbb{C}}\mathbb{C}[G]$ is a vector space with canonical basis $\{\delta_x\otimes g: x\in K,g\in G\}$. The \emph{Naidu quasi-Hopf algebra} is this vector space with quasi-Hopf algebra structure defined by:
\\
\begin{eqnarray}\mathrm{Product:}&
(\delta_x g)(\delta_y h):= \delta_{x,^hy}\gamma_{g,h}(y)^{-1}\delta_y (gh)\\
\mathrm{Unit:}&
1:=\sum_{x\in K}\delta_x \Id\\
\mathrm{Counit:}&
\epsilon(\delta_x g):=\delta_{x,\Id_K}\\
\mathrm{Coproduct:}&
\Delta(\delta_{x}g)=\sum_{a,b\in K, ab=x}\mu_g(a,b)(\delta_a g)\otimes (\delta_b g)\\
\mathrm{Associator:}&
\Phi:=\sum_{x,y,z\in K} \omega(x,y,z)(\delta_x \Id)\otimes (\delta_y \Id)\otimes (\delta_z \Id)\\
\mathrm{Antipode:}&
S(\delta_x g):=\frac{\gamma_{g^{-1},g}(x^{-1})}{\mu_{g}(x,x^{-1)}}\delta_{(x^{-1})^g} g^{-1},\ \alpha:=1,\ \beta:=\sum_{x\in K}\omega(x^{-1},x,x^{-1})\delta_x \Id\\
R\mathrm{-Matrix:}&
R=\sum_{x,y\in K} c(x,y) (\delta_x \Id)\otimes (\delta_y\partial(x))
\end{eqnarray}
\end{definition}
\noindent A Naidu quasi-Hopf algebra will be denoted by $H(\omega,\gamma,\mu,c)$.

\begin{definition}{\cite{MN}}\label{clobj}
A cleft object of $\mathbb{C}_{\omega}^K$ consists of a triple $(G,\sigma,\theta)$ where $G$ is a group acting on the right of $K$ by automorphisms, $\sigma\in C^2(K,(\mathbb{C}^{G})^{\times}),\theta\in C^2(G,(\mathbb{C}^K)^{\times})$ satisfying the following conditions:
\begin{eqnarray}
\theta_{g\cdot x}(y,z)\,\theta_g(x,yz)=\theta_g(xy,z)\,\theta_g(x,y) \,,& g\in K \ x,y,z\in G \label{eqn:cleft_condition_1}\\
\frac{\sigma_x(h,k)\,\sigma_x(g,hk)}{\sigma_x(gh,k)\,\sigma_x(g,h)}=\frac{\omega(g^x,h^x,k^x)}{\omega(g,h,k)} \,,& x\in G \ g,h,k\in K\label{eqn:cleft_condition_2}\\
\frac{\sigma_{xy}(g,h)}{\sigma_x(g,h)\,\sigma_y(g\cdot x,h\cdot x)}=\frac{\theta_g(x,y)\,\theta_h(x,y)}{\theta_{gh}(x,y)} \,,& x,y\in G \ g,h\in K \label{eqn:cleft_condition_3}
\end{eqnarray}
where $\theta_g(x,y)=\theta(x,y)(g), \sigma_x(g,h)=\sigma(g,h)(x)$ for $x,y\in G$ and $g,h\in K$.
\label{cleft_object}
\end{definition}

\begin{definition}\emph{Quantum Cleft Extensions}\label{qclext}

Let $K$ be a group, $\omega\in Z^3(K,\mathbb{C}^{\times})$ and $(G,\theta,\sigma)$ a cleft object of $\mathbb{C}^{\omega}_K$.  The vector space $\mathbb{C}^K\otimes_{\mathbb{C}}\mathbb{C}[G]$ has a canonical basis $\{\delta_x\otimes g:x\in K, g\in G\}$. A \emph{quantum cleft extension} is this vector space with quasi-Hopf algebra structure defined by:
\\
\begin{eqnarray}\mathrm{Product:}&
(\delta_xg)\cdot (\delta_y h)=\delta_{x^g,y}\theta_{x}(g,h)\delta_x (gh)\,, \ \ \ \ x\in K, g,h\in G\\
\mathrm{Unit:}&
1: =\sum_{x\in K}\delta_x \Id_G \\
\mathrm{Counit:}&
\epsilon(\delta_x g):=\delta_{x,\Id_K}\\
\mathrm{Coproduct:}& 
\Delta(\delta_x g):=\sum_{a,b\in K, ab=x}\sigma_g(a,b)(\delta_a g)\otimes (\delta_b g) \,,\ \ \ x\in K, g\in G\\
\mathrm{Associator:}& 
\Phi:=\sum_{x,y,z\in K} \omega(x,y,z)^{-1}(\delta_x \Id_G)\otimes (\delta_y \Id_G)\otimes (\delta_z \Id_G)\\
\mathrm{Antipode:}&
\!\!\!\!\!S(\delta_x g):= \theta_{x^{-1}}(g,g^{-1})^{-1}\sigma_g(x,x^{-1})^{-1}\delta_{(x^{-1})^g}(g^{-1}), \,\\  & \alpha:=1, \,\, 
\beta:=\sum_{x\in K}\omega(x,x^{-1},x)\delta_x \Id_G
\end{eqnarray}
\end{definition}
\begin{example}{Twisted Drinfeld Double}
\\
Fix a group $K$ and a $3$-cocycle $\omega$ of $K$. Consider the following: 
\begin{eqnarray}
\label{double_cleft_1}
\theta_x(g,h):=\frac{\omega(x,g,h)\,\omega(g,h,(gh)^{-1}x(gh))}{\omega(g,g^{-1}xg,h)} \,,& x,g,h\in K
\\ \sigma_g(x,y):= \frac{\omega(x,y,g)\,\omega(g,g^{-1}xg,g^{-1}yg)}{\omega(x,g,g^{-1}yg)} \,, &g,x,y\in K
\end{eqnarray}
One sees that $(K,\theta,\sigma)$ is a cleft object of $\C^K_{\omega}$, and the associated quantum cleft extension is $D^{\omega}(K)$. When we want to emphasize that $D^{\omega}(K)$ comes from a cleft object, we will write $D^{\omega}_{\theta,\sigma}(K)$.
\end{example}

\subsection{Equivalence of Naidu and Mason--Ng quasi-Hopf algebras}
Suppose $(G,K,\partial)$ has a surjective homomorphism $\partial\,{:}\,K\rightarrow G$. $G$ may be considered as a central quotient of $K$ and so acts on it by conjugation. 
\begin{lemma}
Suppose that $\alpha$ is a $3$-cocycle of $G$ with coefficients in $\C^{\times}$, then $\alpha_{\rev}(x,y,z):=\alpha(z,y,x)$ is a $3$-cocycle of $G^{\op}$ with coefficients in $\C^{\times}$.
\end{lemma}

\begin{lemma}
 Let $G$ and $K$ be as before. The map $I\,{:}\,C^2(K,(\mathbb{C}^G)^{\times})\rightarrow C^1(G,C^2(K,\mathbb{C}^{\times}))$ defined by $I(f)(a)(x,y)=f(x,y)(a)$ is an isomorphism of abelian groups.
\label{identification}
\end{lemma}
Let $\overline{z}$ denote complex conjugation for a complex number.
\begin{lemma}
\label{Cleft_Object}
Suppose that $(\alpha,\gamma,\mu,c)$ is a quasi-abelian $3$-cocycle of the $G$-crossed module $(K,G,\partial)$. Then there is an induced cleft object of $\mathbb{C}^{K^{\op}}_{\overline{\alpha}_{\rev}}$ with $G^{\op}$ acting on $K^{\op}$.
\end{lemma}
\begin{proof}
Let $x,y\in K$ then $xyx^{-1}=(y\cdot_{\op} x)\cdot x^{-1}=x^{-1}\cdot_{\op}(y\cdot_{\op} x)$ and so, $y\triangleleft x =\partial(x^{-1})\cdot_{\op} y\cdot_{\op} \partial(x)$ gives a right action of $G^{\op}$ on $K^{\op}$. Let $\overline{\theta}_x(g,h):=\gamma_{h,g}(x)^{-1}$ for $h,g\in G, x\in K$, and $ \overline{\sigma}_g(x,y):=\mu_{g}(y,x)$ for $x,y\in K, y\in G$. It is routine to verify that $(G^{\op},\overline{\theta},\overline{\sigma})$ will satisfy Eqs.\ \eqref{eqn:cleft_condition_1},  \eqref{eqn:cleft_condition_2}, and \eqref{eqn:cleft_condition_3}
\end{proof}
From now on we will fix a quasi-abelian 3-cocycle $(\alpha,\gamma,\mu,c)$ of $(G,K,\partial)$, and denote the quantum cleft extension induced by this as $D^{\overline{\alpha}_{\rev}}_{\overline{\theta},\overline{\sigma}}(K^{\op},A)$, where $A=\ker(\partial)$. When the specific cleft object of $\C^K_{\overline{\alpha}_{\rev}}$ is clear from context, we will use the shorthand $D^{\overline{\alpha}_{\rev}}(K^{\op},A)$.
Let $i\,{:}\,\mathbb{C}^{K{^{\op}}}_{\overline{\alpha}_{\rev}}\rightarrow \mathbb{C}^{K{^{\op}}}_{\overline{\alpha}_{\rev}}\otimes_{\C} \mathbb{C}[G^{\op}] $ be the map determined by $i(\delta_x):=\delta_x\Id_G$ for $g\in K^{\op}$. Similarly, let $p\,{:}\,\mathbb{C}^{K{^{\op}}}_{\overline{\alpha}_{\rev}}\otimes \mathbb{C}[G^{\op}] \rightarrow \mathbb{C}G^{\op}$, as $p(\delta_x g)=\delta_{\Id_K,x}g$ for $x\in K, g\in G$. Here $\delta_{\Id_K, g}$ denotes the Kronecker delta on $K$.
\begin{lemma}
Let $D^{\overline{\alpha}_{\rev}}_{\overline{\theta},\overline{\sigma}}(K^{\op},A)$ be the induced quantum cleft extension from the cleft object in Lemma \ref{Cleft_Object}. This is a quasi-Hopf algebra quotient of $D^{\overline{\alpha}_{\rev}}_{\theta,\sigma}(K^{\op})$.
\label{induced_quasi_quotient_lemma}
\end{lemma}
\begin{proof}
 Let $\partial\,{:}\,K^{\op}\rightarrow G^{\op}$ denote the quotient map. As mentioned in \cite{MN}, to prove this lemma it suffices to find a quasi-bialgebra map $\pi\,{:}\,D^{\overline{\alpha}_{\rev}}_{\theta,\sigma}(K^{\op})\rightarrow D^{\overline{\alpha}_{\rev}}_{\overline{\theta},\overline{\sigma}}(K^{\op},A)$ such that the following diagram commutes: 
\begin{center}
\begin{tikzcd}
\mathbb{C}^{K^{\op}}\arrow[r,"i"]\arrow[d,"\Id"] &D^{\overline{\alpha}_{\rev}}_{\theta,\sigma}(K^{\op})\arrow[r,"p"]\arrow[d,"\pi"]&\mathbb{C} K^{\op}\arrow[d,"\partial"]\\ 
\mathbb{C}^{K^{\op}}\arrow[r,"i"]& D^{\overline{\alpha}_{\rev}}_{\overline{\theta},\overline{\sigma}}(K^{\op},A)\arrow[r,"p"]& \mathbb{C} G^{\op}
\end{tikzcd}
\end{center}
As $D^{\overline{\alpha}_{\rev}}_{\theta,\sigma}(K^{\op})$ has a basis given by $\{\delta_x g\}_{(x,g)\in K\times G}$ and $D^{\overline{\alpha}_{\rev}}_{\overline{\theta},\overline{\sigma}}(K^{\op},A)$ has a basis given by $\{\delta_x \overline{g}\}_{(x,\overline{G})\in K\times G}$, we may define a linear map $\pi$ by $\pi(\delta_x g):=c(g,x)\delta_x\overline{g}$. Here $c$ is the last component coming from the quasi-abelian 3-cocycle $(\alpha,\gamma,\mu,c)$. To see that this makes the diagram commute, it suffices to check that it commutes on the basis. Recall that $(\alpha,\gamma,\mu,c)$ is normal. By definition:
\begin{eqnarray}(\pi \circ i)(\delta_x)=c(\Id_G,x)\delta_x\Id_G=\delta_x\Id_G=i(\delta_x)\Rightarrow \pi\circ i =i \nonumber\\
(p\circ \pi)(\delta_x g )=p(c(g,x)\delta_x\overline{g})=\delta_{\Id_K,x}c(g,x)\overline{g}=\delta_{\Id_K,x}\overline{g}\nonumber\\
(\partial \circ p)(\delta_x g)=\partial(\delta_{\Id_K,x}g)=\delta_{\Id_K,x}\overline{g}\Rightarrow p\circ \pi = \partial\circ p\nonumber\end{eqnarray} This shows that indeed the diagram does commute. 

It suffices now to check that $\pi$ is a quasi-bialgebra morphism. This is true when $\pi$ is an algebra morphism, coalgebra morphism and preserves the associators. We check firstly that it is an algebra morphism: 
\begin{eqnarray}\pi((\delta_x g)(\delta_y h))=\pi(\delta_{x\triangleleft g, y}\,\theta_x(g,h)\,\delta_x(g\cdot_{\op}h))=\delta_{x\triangleleft g, y}\,\theta_x(g,h)(c(g\cdot_{\op} h,x)\,\delta_x(\overline{g}\cdot_{\op}\overline{h}))
\\
\pi(\delta_x g)\cdot \pi(\delta_y h)=c(g,x)\,c(h,y)\,\overline{\theta}_{x}(\overline{g},\overline{h})\,\delta_{x\triangleleft \overline{g},\overline{y}}\,\delta_x(\overline{g}\cdot_{\op} \overline{h})
\end{eqnarray}
Since $A$ is central in $K$, we see that the above two equations are equal if and only if:
\begin{equation}\delta_{gxg^{-1},y}\,c(g,x)\,c(h,y)\,\overline{\theta}_x(\overline{g},\overline{h})=\delta_{gxg^{-1},y}\,c(hg,x)\,\theta_x(g,h)\Leftrightarrow c(g,x)\,c(h,gxg^{-1})\,\overline{\theta}_x(\overline{g},\overline{h})=c(hg,x)\,\theta_x(g,h)\label{eqn:algebra_hom_lemma}\end{equation}
Well, recall that 
\[\theta_x(g,h):=\frac{\overline{\alpha}_{\rev}(x,g,h)\,\overline{\alpha}_{\rev}(g,h,(gh)x(gh)^{-1})}{\overline{\alpha}_{\rev}(g,gxg^{-1},h)}=\Bigg(\frac{\alpha(h,g,x)\,\alpha((gh)x(gh)^{-1},h,g)}{\alpha(h,gxg^{-1},g)}\Bigg)^{-1}\]
This combined with the fact that $\overline{\theta}_x(\overline{g},\overline{h})=\gamma_{\overline{h},\overline{g}}(x)^{-1}$, we see \eqref{eqn:algebra_hom_lemma} is equivalent to 
\begin{equation}
c(hg,x)=\frac{\alpha(h,g,x)\,\alpha((gh)x(gh)^{-1},h,g)}{\alpha(h,gxg^{-1},g)\,\gamma_{\overline{h},\overline{g}}(x)}c(h,gxg^{-1})\,c(g,x)
\end{equation}
This is just \eqref{eqn:qa3c_braiding_1} from Definition \ref{qa3c_definition}. So indeed it is an algebra morphism. Similarly, the fact that $\pi$ is a coalgebra morphism follows immediately from Condition \eqref{eqn:qa3c_braiding_2} of Definition \ref{qa3c_definition}. Lastly, the associator of $D^{\overline{\alpha}_{\rev}}_{\theta,\sigma}(K^{\op})$ is $\Phi:=\sum_{a,b,c\in K^{\op}}\alpha_{\rev}(a,b,c)(\delta_a \Id_K)\otimes (\delta_b \Id_K)\otimes (\delta_c \Id_K)$. This is left unchanged by $\pi$ as $\pi(\delta_a \Id_K)=\delta_a\Id_G$. As $D^{\overline{\alpha}_{\rev}}_{\overline{\theta},\overline{\sigma}}(K^{\op},A)$ is a cleft object of $\mathbb{C}^{K^{\op}}_{\overline{\alpha}_{\rev}}$, this is also its associator. So $\pi$ is a quasi-bialgebra morphism and we are done.
\end{proof}
The fact that $D^{\overline{\alpha}_{\rev}}_{\overline{\theta},\overline{\sigma}}(K^{\op},A)$ is a quasi-Hopf algebra quotient is essential to get a braiding on \textbf{Rep}$\,\, D^{\overline{\alpha}_{\rev}}_{\overline{\theta},\overline{\sigma}}(K^{\op},A)$. For recall that $D^{\overline{\alpha}_{\rev}}_{\theta,\sigma}(K^{\op})$ is a quasitriangular quasi-Hopf algebra with $R$-matrix given by:
\begin{equation}
\label{eqn:R_matrix_drinfeld_double}
R_{\ D^{\overline{\alpha}_{\rev}}_{\theta,\sigma}(K^{\op})}:=\sum_{x,y\in K^{\op}} (\delta_x \Id_K)\otimes (\delta_y x)
\end{equation}
Applying $\pi\,{:}\,D^{\overline{\alpha}_{\rev}}_{\theta,\sigma}(K^{\op})\rightarrow D^{\overline{\alpha}_{\rev}}_{\overline{\theta},\overline{\sigma}}(K^{\op},A)$ we see that $D^{\overline{\alpha}_{\rev}}_{\overline{\theta},\overline{\sigma}}(K^{\op},A)$ has an $R$-matrix given by 
\begin{equation}
\label{eqn:R_matrix_quantum_cleft}
R_{D^{\overline{\alpha}_{\rev}}
_{\overline{\theta},\overline{\sigma}}
(K^{\op},A)}:=\pi(R_{\ D^{\overline{\alpha}
_{\rev}}_{\theta,\sigma}
(K^{\op})})=\sum_{x,y\in K^{\op}}c(x,y) 
(\delta_x\Id_K)\otimes (\delta_y 
\partial(x))
\end{equation}
The astute reader will notice the $R$-matrix  of $D^{\overline{\alpha}_{\rev}}
_{\overline{\theta},\overline{\sigma}}
(K^{\op},A)$ is the $R$-matrix of the quasitriangular quasi-Hopf algebra $H(\alpha,\gamma,\mu,c)$. As mentioned at the start of this section the Mason-Ng description will be the same as Naidu's description of the equivariantization. The slight difference between the two descriptions occurs because the Mason-Ng description utilizes right group actions instead of left actions. Due to this we need to squint at the Mason-Ng description to get the equivalence. The first step to this was noticing that a quasi-abelian 3-cocycle gives a cleft object of the opposite group. The next step we will show is that $D^{\overline{\alpha}_{\rev}}
_{\overline{\theta},\overline{\sigma}}
(K^{\op},A)$ is equivalent as a quasitriangular quasi-Hopf algebra to some twist of $H(\alpha,\gamma,\mu,c)$, with an appropriate alteration of the antipode structure. 

\begin{fact}{\cite[Proposition 10.10]{bulacu_qha}}

Let $H$ be a quasitriangular quasi-bialgebra and $F$ a gauge transformation. Then $H_F$ is also a quasitriangular quasi-bialgebra with $R$-matrix given by $R_F:=F_{21}RF^{-1}$. Here we write $F:=F^1\otimes F^2$ in Sweedler notation so $F_{21}:=F^2\otimes F^1$. 
\end{fact}
\begin{fact}{\cite[Proposition 3.23]{bulacu_qha}\cite[Lemma 2.3]{BuN}}

Let $H$ be a quasi-Hopf algebra. There exists a twist $f\in H\otimes H$ such that $S\,{:}\,H^{\op,\coop}\rightarrow H_f$ is a quasi-Hopf algebra morphism. Moreover, if $H$ is finite dimensional this is an isomorphism. Furthermore, this is a morphism of quasitriangular quasi-Hopf algebras. That is $(S\otimes S)(R)=R_f$. 
\end{fact}
\noindent Combining all these facts we see that if $S$ is an isomorphism, then $H^{\op,\coop}$ will be a quasitriangularizable quasi-Hopf algebra with $R$-matrix given by $R_H$. Combining all of these observations, we obtain the following lemma.

\begin{lemma} \label{qhequiv}

$D^{\overline{\alpha}_{\rev}}
_{\overline{\theta},\overline{\sigma}}
(K^{\op},A)$ as a quasitriangular quasi-Hopf algebra is isomorphic to $H(\alpha,\gamma,\mu,c)^{\op,\coop}$ with  antipode structure twisted by $\beta$.   
\end{lemma}
\begin{proof}
It is clear that the multiplication and comultiplication coincide. Note that the antipode structure of $H(\alpha,\gamma,\mu,c)^{\op,\coop}$ is given by $(S,\beta,1)$, which is why we twist by $\beta$. Furthermore, the associator is given by $\alpha_{\rev}$, but because of our choice of 3-cocycle for $D^{\overline{\alpha}_{\rev}}
_{\overline{\theta},\overline{\sigma}}(K^{\op},A)$ they are the same. Lastly, there is a difference between the $\beta$ of a quantum cleft extension and a Naidu quasi-Hopf algebra, namely  the former has coefficients given by $\overline{\alpha(x,x^{-1},x)}$ and the latter by $\alpha(x^{-1},x,x^{-1})$. But using the fact that $d\alpha(x^{-1},x,x^{-1},x)=1$ we see that these are the same.
\end{proof}

\begin{corollary}
$\mathbf{Rep}\,\, D_{\overline{\theta},\overline{\sigma}}^{\overline{\alpha}_{\rev}}(K^{\op},A)$ is equivalent as a braided fusion category to $\mathbf{Rep}\ H(\alpha,\gamma,\mu,c)$.
\label{the_final_result_of_swapping}
\end{corollary}

As $G^{\op}$ is a quotient of $K^{\op}$ by an abelian group we see that $f\,{:}\,K^{\op}\rightarrow K, f(g)=g^{-1}$ induces an isomorphism of the pairs $(K^{\op},G^{\op})$ and $(K,G)$. We therefore, have a quasi-abelian $3$-cocycle induced by $f$, which we denote by $(\alpha,\gamma,\mu,c)^f$. Denote $\omega:=\overline{\alpha}_{\rev}^f$
\begin{lemma}
$(\alpha,\gamma,\mu,c)^f$ induces a cleft object of $\C^{K}_{\omega}$ $(G,\overline{\theta}^f,\overline{\sigma}^f)$, such that $D^{\omega}_{\overline{\theta}^f,\overline{\sigma}^f}(K,A)$ is a quasi-Hopf algebra quotient of $D^{\omega}(K)$. Furthermore, $D^{\omega}_{\overline{\theta}^f,\overline{\sigma}^f}(K,A)$ is isomorphic as a quasitriangular quasi-Hopf algebra to $D_{\overline{\theta},\overline{\sigma}}^{\overline{\alpha}_{\rev}}(K^{\op},A)$.
\label{switcharoo}
\end{lemma}
\begin{proof}
This proof is routine and left to the reader.
\end{proof}

\begin{theorem}
Let $\mathbf{D}$ be a braided fusion category, containing $\mathbf{Rep}\,\, G$ as a symmetric fusion subcategory. If the de-equivariantization $\mathbf{D}_G$ is pointed and faithfully graded, then let $K$ be the group of isomorphism classes of simple objects, and $A$ the simple objects of the identity component. Suppose that $G$ fixes the isomorphism classes $A$. Then there  exists a quantum cleft extension $D^{\omega}_{\overline{\theta}^f,\overline{\sigma}^f}(K,A)$ such that $\mathbf{D}\cong \mathbf{Rep}\,\, D^{\omega}_{\overline{\theta}^f,\overline{\sigma}^f}(K,A)$ as a braided fusion category. 
\label{main_cherry}
\end{theorem}
\begin{proof}
By \cite{Naidu} we know that there  exists a quasi-abelian $3$-cocycle $(\alpha,\gamma,\mu,c)$ such that $(\mathbf{D}_G)^G\cong \mathrm{\textbf{Rep}}\,\, H(\alpha,\gamma,\mu,c)$ as braided fusion categories. By Lemma \ref{switcharoo} we know there exists a quantum cleft extension  $D^{\omega}_{\overline{\theta}^f,\overline{\sigma}^f}(K,A)$ induced from the quasi-abelian $3$-cocycle. By Corollary \ref{the_final_result_of_swapping}, we see that $\mathbf{D}\cong (\mathbf{D}_G)^G\cong \mathrm{\textbf{Rep}}\,\, D^{\omega}_{\overline{\theta}^f,\overline{\sigma}^f}(K,A)$.
\end{proof}

\subsection{Classification and construction of the  $D^{\omega}(K,A)$ modules}

 In this subsection we use \cite{DPR} to obtain a construction of the irreducible $D^{\omega}(K,A)$ modules.

To describe the isomorphism classes of $D^{\omega}(K)$-modules, we use the set-up from \cite{DPR}.
Given a representation $\rho:D^{\omega}(K)\to \mathrm{End}_{\C}(V)$, define a $K$-grading on $V$ through $V_x:=(\delta_x\otimes id)V$, and a twisted $K$-action through $g\cdot v:=(id\otimes g)v$.  Label the conjugacy classes  of $K$ by $\{C_b\}_{b=1}^r$, and for each $C_b$ choose an element $g_1^b$. Furthermore, for each $b$ choose a set of representatives of $G/C_K(g_1^b)$, where $C_K(g)$ denotes the centralizer, and denote the representatives by $\{x_1^b,\ldots, x_m^b\}$ such that $x_1^b=id$. We then have $C_b=\{g_1^b=x_1^bg_1^b(x_1^b)^{-1},\ldots, g_r^b=x_m^bg_1^b(x_m^b)^{-1}\}$. There is a bijection between the isomorphism classes of irreducible $ D^{\omega}(K)$-modules and irreducible $\mathbb{C}_{\theta_{g_1^b}}[C_K(g_1^b)]$-modules for each $1\leq b\leq r$. Suppose $\rho:\mathbb{C}_{\theta_{g_1^b}}[C_K(g_1^b)]\rightarrow \mathrm{End}_{\C}(V)$ is  irreducible. The corresponding irreducible $D^{\omega}(K)$-module was explicitly constructed in \cite{DPR}, which we recall now.  

Let $\mathcal{B}_b$ be the subalgebra of $D^{\omega}(K)$ which is spanned by the elements $\{\delta_x a:(x,a)\in K\times C_K(g_1^b)\}$. We may define a $\mathcal{B}_b$-representation $\rho_V:\mathcal{B}_b\rightarrow \mathrm{End}_{\C}(V)$ by $\rho_V(\delta_x\otimes a)(v):=\delta_{x,g_1^A}\rho(a)(v)$. One can then take the induced $D^{\omega}(K)$-module, $D^{\omega}(K)\otimes_{\mathcal{B}_b}V$. In \cite{DPR} Dijkgraaf et. al gave an explicit formula for the $D^{\omega}(K)$-action on this induced module which we now explain. First, choose a basis $\{v_t\}_{t=1}^n$ for $V$, and note that a basis of $D^{\omega}(K)\otimes_{\mathcal{B}_b} V$ will be given by $\{(\delta_{\Id_K} x_t^b)\otimes v_j: 1\leq t\leq m, 1\leq j\leq n\}$. The action is then defined on this basis by:
\begin{equation}\rho_V(\delta_x g)((\delta_{\Id_K}x_j^b)\otimes v_{i}):=\theta_x(g,x_j^b)\theta_x(x_k^b,h)^{-1}\delta_{g^{-1}xg,g_j^b}((\delta_{\Id_K}x_k^b)\otimes (\rho(h)(v_i)))\label{eqn:basis_action}\end{equation}
where $x_k,h$ are such that $gx_j^b=x_k^b h$, and $h\in C_K(g_1^b)$.

Now assume that we have a quantum cleft extension such that $D^{\omega}(K,A)$ is a quasi-Hopf algebra quotient with quotient map $\pi:D^{\omega}(K)\rightarrow D^{\omega}(K,A)$ and its modules form a MTC. Since $\mathrm{\textbf{Rep}}\,\, D^{\omega}(K,A)$ is a full fusion subcategory of $\mathrm{\textbf{Rep}}\,\, D^{\omega}(K,A)$, to determine the irreducible objects of $\mathrm{\textbf{Rep}}\,\, D^{\omega}(K,A)$ it suffices to determine which simple objects of $\mathrm{\textbf{Rep}}\,\, D^{\omega}(K)$ are in the fusion subcategory. More precisely, a representation $\rho:D^{\omega}(K)\rightarrow \mathrm{End}_{\C}V$ is in $\mathrm{\textbf{Rep}}\,\, D^{\omega}(K,A)$ iff $\ker(\pi)\subset \ker(\rho)$. For that reason we calculate $\ker(\pi)$. As in the previous subsection, let the quotient map $\pi$ be given by $\pi(\delta_xg)=c(g,x)\delta_x\overline{g}$

\begin{lemma}
Let $\partial:K\rightarrow G$ be as in Lemma \ref{induced_quasi_quotient_lemma}, and $B_{x,h}:=\mathrm{span}_{\C}\{c(g,x)^{-1}\delta_x g-c(k,x)^{-1}\delta_x k: g,k\in \partial^{-1}(h)\}$. Then
\begin{equation}\ker(\pi)=\bigoplus_{(g,h)\in K\times G}B_{g,h}\end{equation}
\label{kernel_description}
\end{lemma}
\begin{proof}
To see this note that if we have choose a set of representatives for $G$, say $\{y_1,\ldots, y_k\}$, then $\{\delta_x\otimes \overline{y}_i: x\in K, 1\leq i\leq k\}$ is a basis of $D^{\omega}(K,A)$. If $v=\sum_{(x,k)\in K^2}a_{x,k}(\delta_{x} k)$, then:
\[\pi(v)=\sum_{(x,h)\in K\times G}\sum_{k\in \partial^{-1}(h)}c(k,x)a_{x,k}\delta_x h=0\Leftrightarrow \sum_{k\in \partial^{-1}(h)}c(k,x)a_{x,k}=0 \ \ \ \ \forall (x,h)\in K\times G\] 
Now fix a pair $(x,h)\in K\times G$, and let $\partial^{-1}(h)=\{h_1,\ldots, h_r\}$. By basic algebra one sees that: 
\[\{\sum_{i=1}^ra_{x,h_i}\delta_x h_i: \sum_{i=1}^rc(h_i,x)a_{x,h_i}=0 \}=\mathrm{span}_{\C}\{c(h_i,x)^{-1}\delta_xh_i-c(h_j,x)^{-1}\delta_x h_j:1\leq i,j\leq r\}=B_{x,h}\] 
\end{proof}

\begin{theorem}
Let $\{g_1^b\}_{b=1}^r$ be a set of conjugacy class representatives for $K$. Suppose that $\rho_V:\C_{\theta_{g_1^b}}[C_K(g_1^b)]\rightarrow \mathrm{End}_{\C}(V)$ is an irreducible representation. Then \begin{equation}(g_1^b,\rho_V)\in \mathbf{Rep}\,\, D^{\omega}(K,A)\text{  if and only if }\rho(a)=c(a,g_1^A)\Id_V \ \ \ \forall a\in A=\ker(\partial)\end{equation}
\end{theorem}
\begin{proof}
First, we show there is a natural bijection between the sets:
\begin{gather*}\left\{ \rho_v \in \mathrm{Irr.Rep}(\mathbb{C}_{\theta_{g_1^b}}[C_K(g_1^b)]) : \rho(a)=c(a,g_1^b)\Id_V \ \ \forall a\in A\right\}\leftrightarrow \mathrm{Irr.Rep}(\mathbb{C}_{\overline{\theta}_{g_1^b}}[C_G(\overline{g}_1^b)])\end{gather*}

Suppose that we have $\rho\in \mathrm{Irr.Rep}(\mathbb{C}_{\theta_{g_1^b}}[C_K(g_1^b)])$ with the prescribed condition. Then observe that $C_G(\overline{g}_1^b)=\overline{C_K(g_1^b)}$, $\frac{1}{c(-,g_1^b)}\rho|_A=\Id$, and 
\[\hspace{-1cm}\frac{1}{c(x,g_1^b)c(y,g_1^b)}\rho(x)\rho(y)=\frac{\theta_{g_1^b}(x,y)}{c(x,g_1^b)c(y,g_1^b)}\rho(xy)=\frac{\theta_{g_1^b}(x,y)c(xy,g_1^b)}{c(x,g_1^b)c(y,g_1^b)}\rho(xy)=\overline{\theta}_{g_1^b}(\partial(x),\partial(y))\rho(xy)\] Here the last equality holds, since $D^{\omega}(K,A)$ is a quotient of $D^{\omega}(K)$, and $x,y\in C_K(g_1^b)$. We see that indeed there will exist a $\tilde{\rho}\in \mathrm{Irr.Rep}(\mathbb{C}_{\overline{\theta}_{g_1^b}}[C_G(\overline{g}_1^b)])$ such that $\tilde{\rho}\circ \partial =\frac{1}{c(-,g_1^b)}\rho$. Conversely, if $\tilde{\rho}:\C_{\overline{\theta}_{g_1^b}}[C_G(\overline{g}_1^b)]
\rightarrow \mathrm{End}_{\C}(V)$ is an irreducible representation, then $\rho(x):=c(x,g_1^b)\tilde{\rho}(\partial(x))$ is a 
irreducible representation of $\mathbb{C}_{\theta_{g_1^b}}[C_K(g_1^b)]$ such that $A\subseteq \ker(\rho)$.  

As outlined above, the irreducible modules of $D^{\omega}(K)$ are determined by a conjugacy class representative  $g_1^b$, and an irreducible projective representation  $\rho$ of its centralizer. To prove the Theorem, we will use the bijection just obtained and a counting argument.

First, note that $(g_1^b,\rho_v)\in \mathrm{\textbf{Rep}}\,\, D^{\omega}(K,A)$ if and only if $\ker(\pi)\subseteq \ker(\rho_V)$. Supposing this is the case, by Lemma \ref{kernel_description} we have $B_{g_1^b,\Id_K}\subseteq \ker(\rho_V)$. Let $x,y\in A$. Then:
 \[\rho_V(c(x,g_1^b)^{-1}\delta_{g_1^b} x)=\rho_V(c(y,g_1^b)^{-1}\delta_{g_1^b} y)\]
Observe that since $A$ is a central subgroup of $K$ we have $xx_1^b=x_1^bx, yx_1^b=x_1^by$ for all $1\leq b\leq m$, and $x_1^b=\Id_K$. Combining these observations we can expand the above action using \eqref{eqn:basis_action} to obtain:
\[
\rho_V(c(x,g_1^b)^{-1}\delta_{g_1^b} x)((\delta_{\Id} x_1^b)\otimes v_j)=c(x,g_1^b)^{-1}\theta_{g_1^b}(x,x_1^b)\theta_{g_1^b}(x_1^b,x)^{-1}(\delta_{\Id_K} x_1^b)\otimes \rho_V(x)(v_j)\]
 \[\rho_V(c(g_1^b,y)^{-1}\delta_{g_1^b} y)((\delta_{\Id} x_1^b)\otimes v_j)=c(y,g_1^b)^{-1}\theta_{g_1^b}(y,x_1^b)\theta_{g_1^b}(x_1^b,y)^{-1}(\delta_{\Id_K} x_1^b)\otimes \rho_V(y)(v_j)\]
 But, since $x_1^b=1$, we see that 
 \[\theta_{g_1^b}(y,x_1^b)\theta_{g_1^b}(x_1^b,y)^{-1}=1\Rightarrow c(x,g_1^b)^{-1}\rho_V(x)=c(y,g_1^b)^{-1}\rho_V(y)
\] In particular, by setting $y=\Id_K$ we have :
\[c(x,g_1^b)^{-1}\rho_V(x)=\rho(\Id_K)=\Id_V\Rightarrow \rho_V(x)=c(x,g_1^b)\Id_V \ \ \ \ \forall x\in A\]

We see then that every simple $(g_1^b,\rho_V)\in \mathrm{\textbf{Rep}}\,\, D^{\omega}(K,A)$ corresponds to an irreducible representation of $\mathbb{C}_{\overline{\theta}_{g_1^b}}[C_G(\overline{g}_1^b)]$. We claim that it is surjective. 

Denote the simple modules of $D^{\omega}(K,A)$ by $\mathcal{O}$.
 Since $\mathrm{\textbf{Rep}}\,\, D^{\omega}(K,A)$ is integral we have that $\dim_{\C}(D^{\omega}(K,A))=|K|\cdot|G|=\sum_{V\in \mathcal{O}}\dim_{\C}(V)^2$. By looking at 
the projective analogue of the regular representation for $\C_{\overline{\theta}_{g_1^b}}[C_G(\overline{g}_1^b)]$ we know that \cite[Proposition 2.3]{Cheng}: 
\[|C_G(\overline{g}_1^b)|=\sum_{W\in \mathrm{Irr}(\C_{\overline{\theta}_{g_1^b}}[C_G(\overline{g}_1^b)])}\dim_{\C}(W)^2\]
Suppose there exists a $1\leq b_0\leq r$ and a $V\in \mathrm{Irr}(\C_{\overline{\theta}_{g_1^{b_0}}}[C_G(\overline{g}_1^{b_0})])$ that does not correspond to a simple object of $\mathrm{\textbf{Rep}}\,\, D^{\omega}(K,A)$. Note that the dimension of $(g_1^b,\rho_V)$ is $\frac{|K|}{|C_K(g_1^b)|}\dim(V_{\rho})$, but by assumption this implies that:
\begin{equation}|G|\cdot |K|<\sum_{b=1}^{r}\frac{|K|^2}{|C_K(g_1^{b})|^2}\sum_{W\in \mathrm{Irr}(\C_{\overline{\theta}_{g_1^{b}}}[C_G(\overline{g_1}^{b})]) }
\dim(W)^2=\sum_{b=1}^{r}\frac{|K|^2}{|C_K(g_1^{b})|^2}|C_G(\overline{g}_1^{b})|= \end{equation}
\[\sum_{b=1}^{r}\frac{|K|^2}{|C_K(g_1^{b})|^2}\frac{|C_K(g_1^{b})|}{|A|}=\frac{|K|^2}{|A|}\sum_{b=1}^{r}\frac{1}{|C_K(g_1^{b})|}=\frac{|K|^2}{|A|}\cdot 1 = |G|\cdot |K|\]
This is a contradiction, which proves the correspondence is surjective, implying the theorem.

 \end{proof}
 Notice that since $G$ is a central quotient of $K$, the set of representatives of the action of $G$ on $K$ is just given by $\{g_1^b\}_{b=1}^r$. In other words, it is really just the conjugation action of $G$. With this in mind note (compare with \cite{Naidu}):
 \begin{corollary} Up to equivalence,
 the simple objects of $\mathbf{Rep}\,\, D^{\omega}(K,A)$ are parametrized by the irreducible projective representations of $\C_{\overline{\theta}_{g_1^b}}[C_G(g_1^b)]$ for each $1\leq b\leq r$.
 \end{corollary}
We already obtained this classification in Section 4.4. The advantage of the approach here is that we can construct all irreducible $D^{\omega}(K,A)$ modules. Namely, if $\rho_V$ is as above, then $\tilde{\rho}_V(\delta_g\overline{x}):=\frac{1}{c(x,g)}\rho_V(\delta_gx)$. Explicitly, 
  \begin{equation}\tilde{\rho}_V(\delta_g\overline{x})((\delta_{\Id_K} x_j^b)\otimes v_i)=\frac{1}{c(x,g)}\frac{\theta_g(x,x_j^b)}{\theta_g(x_k^b,h)}\delta_{g,xg_j^bx^{-1}}(
\delta_{\Id_K}x_k^b)\otimes \rho_V(h)(v_i)\end{equation}
Here again $x_j^b,x_k^b, h$ are all determined by $xx_j^b=x_k^bh, h\in C_K(g_1^b)$.

\newcommand\biba[7]   {\bibitem{#1} {#2:} {\sl #3.} {\rm #4} {\bf #5,}
                    {#6 } {#7}}
                    \newcommand\bibx[4]   {\bibitem{#1} {#2:} {\sl #3} {\rm #4}}
\vspace{0.2cm}\addtolength{\baselineskip}{-2pt}
\begin{footnotesize}
\noindent{\it Acknowledgement.} We thank Simon Lentner and Sven M\"oller for discussions. The research of TG was supported in part by an   NSERC Discovery Grant. This material is based upon work supported by the National Science Foundation under Grant No. DMS-1928930, while TG was in residence at the Mathematical Sciences Research Institute in Berkeley, California, during the Quantum Symmetries Reunion of 2024. AR was supported in part by an NSERC CGS M.

\end{footnotesize}

\end{document}